\documentclass[12pt,a4paper]{article}

\usepackage[top=3cm, bottom=3.5cm, left=2cm, right=2cm]{geometry}
\usepackage[T1]{fontenc}
\usepackage[]{geometry}
\usepackage{graphicx}
\usepackage{mathtools}
\usepackage{amssymb}
\usepackage{amsthm}
\usepackage{thmtools}
\usepackage{xcolor}
\usepackage{nameref}
\usepackage[english]{babel}
\usepackage{hyperref}
\usepackage{enumitem}
\usepackage{tikz}
\usepackage{ytableau}

\usetikzlibrary{arrows.meta,positioning,calc,fit,decorations.pathmorphing,shapes.callouts}

\definecolor{myblue}{HTML}{1E5BAE}
\tikzset{
	box/.style={
		draw=myblue, very thick, rounded corners,
		inner sep=6pt, minimum width=6.6cm, minimum height=3.2cm,
		align=left
	},
	arr/.style={-Stealth, very thick, myblue},
	darr/.style={<->, very thick, myblue},
	note/.style={myblue, font=\bfseries\large}
}

\usepackage[utf8]{inputenc}

\usepackage{tabularx}
\usepackage[numbers]{natbib}
\usepackage{pgfplots}
\usetikzlibrary{arrows.meta}
\usepackage{framed}

\usepackage{longtable}
\usepackage{fancybox}
\usepackage{float}
\usepackage{pifont}
\usepackage{pdfpages}
\usepackage{calc}
\usepackage{varioref}
\usepackage{pgf}
\usepackage{graphicx}
\usepackage{varioref}

\usepackage{titling}
\usepackage{mathrsfs}
\usetikzlibrary{arrows}
\usetikzlibrary{plotmarks}
\usetikzlibrary{positioning}
\usetikzlibrary{automata}
\usetikzlibrary{decorations.markings}
\usetikzlibrary{shapes.arrows}
\usepackage{amsthm}
\usepackage{fancyhdr}
\usepackage[explicit]{titlesec}
\usepackage{lmodern}
\usepackage{lipsum}
\usepackage{color}
\newtheorem{theorem}{Theorem}[section]
\newtheorem{corollary}{Corollary}[section]
\newtheorem{prop}{Proposition}[section]

\newtheorem{definition}{Definition}[section]
\newtheorem{example}{Example}[section]

\theoremstyle{definition} 
\newtheorem{remark}{Remark}[section]

\theoremstyle{remark}

\usepackage{lmodern}
\usepackage{array}
\usepackage{caption}
\usepackage[none]{hyphenat}
\usepackage{float}

\newcommand{\suit}[1]{{#1}(n,k)}

\newcommand{\we}{A_{m,r}(n,k)}
\newcommand{\ew}[2]{A_{m,r}(#1,#2)}

\newcommand{\coul}[1]{#1_{n,k}}

\newcommand{\opd}[1]{\mathcal{#1}}

\newcommand{\idc}{\in \natu}

\newcommand{\agenm}[2]{(#1)^{(\underline{#2})}}

\newcommand{\natu}{\mathbb{N}}

\newcommand{\varn}{\frac{z^n}{n!}}

\newcommand{\van}{\frac{t^n}{n!}}
\newcommand{\vatn}{\frac{t^n}{n!}}

\newcommand{\nkid}{_{n,k\idc}}

\begin{document}

	\begin{flushleft}
		\centering
		
	\end{flushleft}
	\begin{center}
		
		{\huge\bf{  Triangular Arrays using context-free grammar }}
		
		\vspace{0.5cm}
		\textbf{AUBERT Voalaza Mahavily Romuald} \\
		\texttt{aubert@aims.ac.za} \\
		Department of Mathematics, Faculty of Science \\
		Laboratoire de Mathématiques et Applications de l’Université de Fianarantsoa (LaMAF) \\
		University of Fianarantsoa \\
		Postal Adress 601, Fianarantsoa, Madagascar
		
		\vspace{0.5cm}
		
		\textbf{RANDRIANIRINA Benjamin} \\
		\texttt{benjamin.randrianirina@univ-fianarantsoa.mg} \\
		Department of Mathematics, Faculty of Science \\
		Laboratoire de Mathématiques et Applications de l’Université de Fianarantsoa (LaMAF) \\
		University of Fianarantsoa \\
		Postal Adres 601, Fianarantsoa, Madagascar
	\end{center}

	\section*{Abstract} 
	
	In this work, the Hao grammar $G=\{\, u\rightarrow u^{b_1+b_2+1} v^{a_1+a_2},\quad 
	v\rightarrow u^{b_2}v^{a_2+1} \,\},$
	together with the correspondence between grammars and combinatorial differential equations, is employed to obtain an interpretation of any triangular array of the form
	\[
	T(n,k)=(a_2 n + a_1 k + a_0)\,T(n-1,k)
	+ (b_2 n + b_1 k + b_0)\,T(n-1,k-1).
	\]
	This lead to have an interpretation of $T(n,k)$ as an increasing tree.
	Explicit formulas and structural properties are then derived through analytic differential equations.  
	In particular, the $r$-Whitney-Eulerian numbers and the cases where $b_2n+b_1k+b_0=1$ are obtained explicitly.
	
	\noindent Applications include new interpretation formulas for the $r$-Eulerian numbers with generating functions. We also obtain full generating functions for the case $a_2=-a_1$ using this approach.
	
	\noindent
	\textbf{Keywords:} triangular recurrence, formal grammar,  differential equations, $r$-Eulerian, combinatorial interpretation, $r$-Whitney--Eulerian.

	\section*{Introduction }
	In mathematics, certain sequences of numbers satisfy a triangular recurrence of the form 
	
	\begin{equation}\label{recu}
		\suit{T}=(a_0+a_1k+a_2n)T(n-1,k)+(b_0+b_1k+b_2n)T(n-1,k-1).
	\end{equation}
	
	\noindent After discussing the binomial coefficients, the Stirling numbers, and the Eulerian numbers, Graham et all proposed a generalization problem of the form \eqref{recu} in  \cite{GrahamKnuthPatashnik1994}. 
	In combinatorics, several approaches have been developed regarding these kinds of numbers: the GKP numbers\footnote{Graham, Knuth, and Patashnik.}. One of the results is due to Neuwirth \cite{NEUWIRTH200133}\footnote{It is what Spivey(\cite{Spivey2011JIS} said. We check this paper: it is not there, maybe in another literature).}, who obtained an explicit formula for the case $b_2=0$ using the Galton triangle. Spivey \cite{Spivey2011JIS} found several cases using finite differences. Analytical approaches have also appeared in various works: Théorêt \cite{theoret1994hyperbinomiales}, Wilf \cite{wilf2004methodcharacteristicsproblem89}, Barbero \cite{BARBEROG2014146}.
	
	Grammatical approaches and interpretations have also already been developed by Hao and al. \cite{hao2015context}, Zhou and al. \cite{Zhou2022}.
	To study the real-rootedness of polynomial $T_n(x)=\sum_k T(n,k) x^k$, Hao (\cite{hao2015context}) introduced the grammar $G=\{\, u\rightarrow u^{b_1+b_2+1} v^{a_1+a_2},\quad 
	v\rightarrow u^{b_2}v^{a_2+1} \,\}$ and proved that 
	\[\opd{G}^n(u^{b_0+b_1+b_2}v^{a_0+a_2})= \sum_{k=0}^n \suit{T} u^{b_2 n+b_1 k +b_0+b_1+b_2} v^{a_2 n + a_1 k +a_0+a_2}.\]
	On another hand, the context-free grammar, by Chen (\cite{CHEN1993113}), developed by Dumont (\cite{dumont1996william}) was continued by  Randrianirina (\cite{Randrianirina2000}) with more species interpretation. These last authors showed the relations between the  grammar relation and a system of differential equations.
	This approach say that we can associate the grammar with a system of differential equations 
	\begin{equation*}
		\left\{
		\begin{array}{ll}U'= U^{b_1+b_2+1}V^{a_1+a_2},& U(0)=u\\ V'= U^{b_2}V^{a_2+1}&
			V(0)=v;\\
		\end{array}\right.
	\end{equation*}
	and the solution of this system verifies
	\begin{equation*}
		\left\{
		\begin{array}{ll}U(t)= Gen(u,t)=\sum_{n\geq 0}\mathcal{G}^n(u)\frac{t^n}{n!}\\ V(t)= Gen(v,t)=\sum_{n\geq 0}\mathcal{G}^n(v)\frac{t^n}{n!}.
		\end{array}\right.
	\end{equation*}
	
	\noindent In this work, we exploit and combine the grammar of Hao with this fact to study the sequence of triangular recurrence.	
	The recurrence \eqref{recu} generally arises in the enumeration of combinatorial structures and objects. Moreover, the associated grammar can provide new combinatorial interpretations of the numbers $\big(\suit{T}\big)_{n,k\ge 0}$, as Ramírez \cite{ramirezarxiv} did for the $r$-Whitney numbers.
	
	We consider an initial conditions usually taken as
	\begin{equation}\label{condiusue}
		T(0,0)=1, \quad \suit{T}=0 \text{ if } n<\max(k,0).
	\end{equation}
	
	\noindent
	We remain in the setting where $a_0+a_1k+a_2n\not=0$ and $b_0+b_1k+b_2n\not=0$. It is more convenient for us to work in the case where $b_0+b_1k+b_2n=1$.  However, generalizations of Eulerian numbers where $b_0+b_1k+b_2n\neq 1$ are also important examples when discussing sequences satisfying \eqref{recu}.
	
	This task begins by recalling the background needed to understand the resolution of combinatorial systems of differential equations. We then present the method we will use, which combines context-free grammars with systems of differential equations. Finally, we conclude with several applications.
	
	\section{Recalls}
	In this section, we recall the notions required for understanding this work. These include the solution of combinatorial differential equations.
	A complete and detailed account of species theories may be found in Bergeron and all \cite{bergeron1998combinatorial}.
	The methods for solving combinatorial differential equations and systems of combinatorial differential equations are presented by Leroux and Viennot in \cite{leroux-viennot} (see also Bergeron and all \cite{bergeron1998combinatorial} Chapter 5). Further developments on $\mathbb{L}$-species, mixed species, and the analysis of initial conditions different from $0$ and $1$ are given by Randrianirina in \cite{randrianirina2000combinatoire} (see also Randrianirina \cite{Randrianirina2000}).

	According to André Joyal (\cite{joyal1981theorie}, definition 1), a species of structure, also known as a $\mathbb{B}$-species, is a functor from the category $\mathbb{B}$ of finite sets and bijection to the category $\mathbb{B}$. A linear species or $\mathbb{L}$-species (André Joyal \cite{joyal1981theorie}, definition 12) is a functor from the $\mathbb{L}$ of  finite linear orders and order-preserving bijection to the category $\mathbb{B}$.
	The exponential generating series of a species $\mathcal{F}$ ($\mathbb{B}$ or $\mathbb{L}$) is the series $\mathcal{F}(t) = \sum_{n\geq 0}f_n \frac{t^n}{n!}$,
	where $f_n = |\mathcal{F}[n]$|.
	
	Operations sum, product, composition, and derivations that are compatible with the
	transition to exponential generating series are defined on the class of species ($\mathbb{B}$ or $\mathbb{L}$).
	
	\noindent
	However, integration
	is only possible for $\mathbb{L}$-species. If $\mathcal{F}$ is an $\mathbb{L}$-species, then the integral of $\mathcal{F}$, denoted $\int \mathcal{F}$, is defined for every finite totally ordered set $l$ by
	$(\int\mathcal{F})(l) = \emptyset$ if $l = \emptyset$, and $( \mathcal{F})(l) = \mathcal{F}(l \setminus \{\min l\})$ if $l \neq\emptyset$.
	A mixed species is a functor from the category $\mathbb{L}\times \mathbb{B}$ to the category of  finite  sets and  bijection.
	
	The method for solving systems of differential equations is essentially due to P. Leroux and G.X. Viennot in \citealp{leroux-viennot}).
	In full generality, the data of an $\mathbb{L}$-species or any $\mathbb{B}$-species $\mathcal{F}$ allow us to write the combinatorial differential equation:
	\begin{equation}
		Y' = \mathcal{F}(Y) ;\quad Y(0) = Z
		\label{1.1}
	\end{equation}
	where $Z$ is a sort of point representing the initial condition.
	This	equation  may be written in integral form:
	\begin{equation}
		Y(T,Z) = Z + \int_{0}^{T} \mathcal{F}(Y(X,Z))\, dX
		\label{1.3}
	\end{equation}
	This integral equation  is interpreted by Figure \ref{fig:(1.1)}, which is an iterative process for constructing the combinatorial solution of equation (\ref{1.1}). 
	\begin{figure}[H]
		\begin{tikzpicture}[line cap=round,line join=round,>=triangle 45,x=1cm,y=1cm]
			\clip(-5,0) rectangle (24.17,3);
			\draw [-] (0,0)-- (0,3);
			\draw [-] (0,3)-- (6,3);
			\draw [-] (6,3)-- (6,0);
			\draw [-] (6,0)-- (0,0);
			\draw [-] (0.53,0.71) circle (0.32249030993194205cm);
			\draw [-] (2.53,2.01) circle (0.31048349392520064cm);
			\draw [-] (3.51,2.13) circle (0.32249030993194205cm);
			\draw [-] (4.33,2.11) circle (0.30594117081556726cm);
			\draw [-] (5.17,1.83) circle (0.3231098884280701cm);
			\draw [-] (3.71,0.71)-- (4.932577740989851,1.6108409916829391);
			\draw [-] (3.71,0.71)-- (2.654447632017464,1.725548269674368);
			\draw [-] (3.71,0.71)-- (3.51,1.8075096900680578);
			\draw [-] (3.71,0.71)-- (4.21637630164624,1.8259407541156003);
			\draw [shift={(3.71,0.71)},line width=0.5pt]  plot[domain=0:2.825976991551027,variable=\t]({1*0.54*cos(\t r)+0*0.54*sin(\t r)},{0*0.54*cos(\t r)+1*0.54*sin(\t r)});
			\draw (4.9,2.1) node[anchor=north west] {$Y$};
			\draw (4.07,2.37) node[anchor=north west] {$Y$};
			\draw (3.25,2.38) node[anchor=north west] {$Y$};
			\draw (2.27,2.27) node[anchor=north west] {$Y$};
			\draw (0.25,1.) node[anchor=north west] {$Y$};
			\draw (2.27,0.8) node[anchor=north west] {$ou$};
			\draw (3.31,0.6) node[anchor=north west] {$min$};
			\draw (1.55,0.5) node[anchor=north west] {$z$};
			\draw (4.35,1.21) node[anchor=north west] {$\mathcal{F}$};
			\draw (1.07,0.8) node[anchor=north west] {=};
			\begin{scriptsize}
				\draw [fill=black] (3.71,0.71) circle (2.5pt);
				\draw [color=black] (1.73,0.67) circle (3.5pt);
			\end{scriptsize}
		\end{tikzpicture}
		\caption{Integral equation}
		\label{fig:(1.1)}
	\end{figure}
	
	The general solution is the increasing $\mathcal{F}$-enriched tree $A_F(T,Z)$, who is a $\mathbb{L}$-species if $\mathcal{F}$ is a $\mathbb{L}$-species, and a mixed species if $\mathcal{F}$ is a $\mathbb{B}$-species. We have:
	\begin{equation}
		A_F(T,Z)=\exp(T\mathcal{D})(Z)=\sum_{n\geq 0}\frac{(T\mathcal{D})^n}{n!}(Z)
		\label{equa1.11}
	\end{equation}
	where $\mathcal{D}=\mathcal{F}(Z)\frac{d}{dZ}$ is the combinatorial differential operator associated with (\ref{1.1}).
	
	An analytical initial condition of the form $y(0) = x$ is combinatorially translated as $Y (0) = 1_x$ , where $1_x$ is the species of the empty set, weighted by $x$. If $\mathcal{F}$ is a $\mathbb{B}$-species, the combinatorial differential equation $Y' = \mathcal{F}(Y), Y(0) = 1_x$  makes sense and its solution is
	the weighted $\mathbb{L}$-species $\mathbb{A}_\mathcal{F}(T ) = T_{X=x}({A}_\mathcal{F} (T,X))$ of the types with respect to the variable $X$ of ${A}_\mathcal{F} (T, X)$ (see Randrianirina \cite{randrianirina2000combinatoire} theorem 6.3). Its generating series $y(t) = \mathbb{A}_\mathcal{F}(t)$ is generally the solution of the differential equation:
	$$y'(t) = Z_\mathcal{F} (y(t); x, x^2 ,x^3,\cdots),~~y(0) = x$$ 
	where $Z_\mathcal{F}$ is the cycle index series of $\mathcal{F}$. If $\mathcal{F}$ is asymmetric, this equation becomes:
	$$y'(t) = \mathcal{F}(y),~~y(0) = x.$$
	These results can be generalized to the system of combinatorial differential equations:
	\begin{equation}
		Y_i' = \mathcal{F}_i(Y_1, Y_2, \ldots, Y_k), \quad Y_i(0) = Z_i,\quad 1 \leq i \leq k
		\label{1.5}
	\end{equation}
	The solution of system (\ref{1.5}) is the $k$-tuple of $\overrightarrow{\mathcal{F}}$-enriched increasing trees $\overrightarrow{\mathcal{A}}_{\overrightarrow{\mathcal{F}}}=(\mathcal{A}_{{\overrightarrow{\mathcal{F}}},1},\cdots,\mathcal{A}_{{\overrightarrow{\mathcal{F}}},1})$, where $\mathcal{A}_{{\overrightarrow{\mathcal{F}}},i}$ is the solution of the equation $Y_i' = \mathcal{F}_ i (Y_1 ,\cdots , Y_k ), Y_i (0) = X_i$. 
	Throughout the remainder of this paper, we assume that each $\mathcal{F}_i$	is an asymmetric $\mathbb{B}$-species.
	
	According to  Chen \cite{CHEN1993113} (see also Dumont \cite{dumont1996william}) a context-free grammar $G$ on $X$ is a map from $X$ to $\mathbb{C}[X]$, where $X=\{x_1,x_2,\cdots\}$ be an alphabet and $\mathbb{C}[X]$ a commutative algebra of polynomials in the letters $x_i$. 
	
	For each grammar $G$, we associate a differential operator
	\[
	\opd{G}=\sum_{x\in X} G(x) \frac{\partial}{\partial x}.
	\]
	satisfying $\mathcal{G}(f+g)=\mathcal{G}(f)+\mathcal{D}(g)$ and $\mathcal{G}(fg)=\mathcal{G}(f)g+f\mathcal{G}(g)$, for all $g\in\mathbb{C}[X]$.
	
	For $u\in C[X]$, we associate an exponential generating function
	\[
	\mathrm{Gen}(u,t):=\sum_{n\ge 0} \opd{G}^n(u)\,\frac{t^n}{n!}.
	\]
	Thus, for all $u,v\in C[X]$,
	\begin{equation}\label{algG}
		\begin{array}{c}
			\mathrm{Gen}(u+v,t)=\mathrm{Gen}(u,t)+\mathrm{Gen}(v,t),~~
			\mathrm{Gen}(uv,t)=\mathrm{Gen}(u,t)\mathrm{Gen}(v,t)\\  
			\quad\text{and}\quad \frac{\partial}{\partial t}\mathrm{Gen}(u,t)
			= \opd{G}\big(\mathrm{Gen}(u,t)\big).
		\end{array}
	\end{equation}
	Let $X=\{x_1,x_2,\cdots,x_k\}$ be the alphabet. We have the following proposition, announced by Dumont in \cite{dumont1996william} and proved by Randrianirina in \cite{randrianirina2000combinatoire} 
	(see also \cite{Randrianirina2000}):
	\begin{prop}\label{chendumrabez}
		Let $\vec{y}(t)=(y_i(t))_{i=1,\ldots,k}$ be the solution of the analytic differential system
		\begin{equation}
			y_i'(t) = G_i(\vec{y}(t)), \qquad y_i(0)=x_i, \quad i=1,\ldots,k.
			\label{equa6}
		\end{equation}
		Then for each $i$, $\mathrm{Gen}(x_i,t)=y_i(t)$.
	\end{prop}
	The following theorem is proved by Randrianirina in 
	\cite{randrianirina2000combinatoire}:
	
	\begin{theorem}\label{rabeza}
		Given asymmetric species $(G_i)_{i=1,\ldots,k}$, the following data are equivalent:
		\begin{enumerate}
			\item the combinatorial differential system
			\begin{equation}
				Y_i' = G_i(\vec{Y}); \qquad Y_i(0)=X_i, \quad i=1,\ldots,k;
				\label{equa71}
			\end{equation}
			\item the associated combinatorial differential operator
			\begin{equation}
				\opd{G} = \sum_{i=1}^{k} G_i(\vec{X}) \frac{\partial}{\partial X_i};
				\label{equa81}
			\end{equation}
			\item William Chen grammar, where $G_i(\vec{x})$ is the generating series of $G_i$: 
			\begin{equation}
				G=\{ x_i \mapsto G_i(\vec{x}); \qquad i=1,\ldots,k\};
				\label{equa91}
			\end{equation}
			\item the analytic differential system
			\begin{equation}
				y_i' = G_i(\vec{y}(t)); \qquad y_i(0)=x_i, \quad i=1,\ldots,k.\label{equa100}
			\end{equation}
		\end{enumerate}
	\end{theorem}
	\noindent
	The combinatorial differential operator $\mathcal{G}$ makes it possible to construct the combinatorial solutions of system (\ref{equa71}). The solutions $(y_i(t))$ of the system of analytic differential equations (\ref{equa100}) are the generating series of the isomorphism types of these solutions. And for all $i\in[k]$, $y_i(t)=\mathrm{Gen}(x_i,t)$.
	
	\section{GPK numbers and the grammar of Hao et al.}
	
	Let us consider the sequence \(\big(\suit{T}\big)_{n,k}\) satisfying (\ref{recu}):
	\begin{equation*}
		\suit{T} = (a_2 n + a_1 k + a_0) T(n-1,k) + (b_2 n + b_1 k + b_0) T(n-1,k-1).
	\end{equation*}
	\noindent
	Let \(G\) be the grammar
	\begin{equation}
		G = \{\, u \rightarrow u^{b_1+b_2+1} v^{a_1+a_2}, \;\;
		v \rightarrow u^{b_2} v^{a_2+1} \,\}, \label{gram2}
	\end{equation}
	so that the associated differential operator is
	\[
	\mathcal{G} = u^{b_1+b_2+1} v^{a_1+a_2} \frac{\partial}{\partial u} + u^{b_2} v^{a_2+1} \frac{\partial}{\partial v}.
	\]   
	\noindent
	Hao et al.~\cite{hao2015context} proved the following result.
	\begin{prop} \label{phaores}
		if $a_1\geq 0$, $a_1+a_2\geq 0$, $a_1+a_3\geq 0$, $b_1\geq 0$, $b_1+b_2\geq 0$ and $b_1+b_2+b_3\geq 0$, then:
		\begin{equation}\label{haores}
			\opd{G}^n(u^{b_0+b_1+b_2}v^{a_0+a_2})= \sum_{k=0}^n \suit{T} u^{b_2 n+b_1 k +b_0+b_1+b_2} v^{a_2 n + a_1 k +a_0+a_2}.
		\end{equation}
	\end{prop}

\begin{example}[The $r$-Whitney-Eulerian numbers $\we$]\rm In Foata \cite{foata2005theoriegeometriquedespolynomes}, ${}^rA(n,k)$ is defined as the number of $\sigma \in S_n$ having $k$ $r$-excedances, where $j \in [n]$ is an $r$-excedance of $\sigma$ if $j+r \leq \sigma(j)$. Here we use the notation $A_r(n,k)$. The author proved that these numbers satisfy the recurrence relation:
\begin{equation}\label{ramr}
	A_r(n,k)=(k+r) A_r(n-1,k) +(n-k+1-r) A_r(n-1,k-1).   
\end{equation}

\noindent	Riordan (\cite{Riordan1958IntroductionTC}) (also in Maier \cite{Maier_2023}) discussed interpretations of the numbers satisfying relation \eqref{ramr}, such as statistics of $r$-descents. This is a generalization of the Eulerian numbers $A_r(n,k)=\left\langle {n \atop k} \right\rangle$. It is also the number of $\sigma \in S_n$ having $n-r-k$ indices $i \in [n-1]$ satisfying $\sigma(i) < \sigma(i+1)$ and $\sigma(i) < n-r$.
To generalize, first, we neglect the fact that this definition requires $r<n$. We start with $r$ completely arbitrary. Second, we consider that each $i \in [n]$ may have $m$ different types. These ideas inspire the definition of the $r$-Whitney-Eulerian numbers $\we$. A combinatorial interpretation is given in Thamrongpairoj \cite{Thamrongpairoj1997}.

\noindent	These numbers are given by the recurrence relation ($\we=0$ if $n<\max(0,k)$):
\begin{equation*}
	\ew{0}{0}=1 \quad \text{and} \quad 
	\we= (mk+r) \ew{n-1}{k}+ \big(mn-mk+m-r  \big) \ew{n-1}{k-1}.
\end{equation*}

\noindent	The grammar of Hao associated is then 
\begin{align*}
	G=\{u&\to uv^m, \\
	v&\to u^mv \};
\end{align*} 
and the system of differential equations is
\begin{align}
	&\left\{\begin{array}{cl}
		U'= & UV^m, \quad U(0)=u,  \\
		V'= &U^mV, \quad V(0)=v.
	\end{array}\right. \label{001} 
\end{align}
First, let's see the analytic solution
\begin{align}
	&\left\{\begin{array}{cl}
		U(t)=\Big(\dfrac{(u^m-v^m)u^m}{u^m-v^m e^{(u^m-v^m)mt}}   \Big)^{\frac{1}{m}},  \\
		V(t)=\Big(\dfrac{(v^m-u^m)v^m}{v^m-u^m e^{(v^m-u^m)mt}}   \Big)^{\frac{1}{m}}.
	\end{array}\right.   \label{eq112}
\end{align}
Then, we have
\begin{align*}
	Gen(u^{m-r}v^r,t)=&\sum_{n\idc} \opd{G}^n(u^{m-r}v^r) \van=\bigg( U(t) \bigg)^m \times \bigg(\frac{V(t)}{U(t)}\bigg)^r.
\end{align*}
So the first thing that Theorem \ref{rabeza} and
the  Proposition \ref{phaores} give us 

\begin{align}
	\sum_{n\idc} \Big( \sum_{k=0}^{n} \we u^{mn-mk} v^{mk} \Big) \van=& \dfrac{(u^m-v^m)e^{(u^m-v^m)rt}}{u^m-v^m e^{(u^m-v^m)mt}}.
\end{align}
This equation can be used to get many things. For instance
\begin{equation}\label{sumamr}
	\sum_{k=0}^{n} \we =m^n n!;
\end{equation}

\begin{equation}
	\sum_{k=0}^{n}(-1)^k \we =2^n \sum_{k} \binom{n}{k} m^k E_k(0) r^{n-k};
\end{equation}

\begin{equation}\label{amrnk}
	\sum_{k\idc} \Big( \sum_{n\idc}\we \van\Big) x^k= \frac{(1-x)e^{r(1-x)t}}{1-xe^{(1-x)mt}}.
\end{equation}
\end{example} \noindent
We can get also the explicit formula of $\we$. Indeed,
\begin{align}
\frac{(1-x)e^{r(1-x)t}}{1-xe^{(1-x)mt}}=&(1-x)\sum_{j\geq 0} x^j e^{(1-x)(mj+r)t} \notag \\
=&(1-x)\sum_{j\geq 0} x^j \sum_{n\geq 0} \frac{\big((1-x)(mj+r)t\big)^n}{n!}  \notag \\
=&\sum_{n\geq 0} \Big(  (1-x)^{n+1} \sum_{j\geq 0} x^j (mj+r)^n   \Big) \van   \notag \\
=&\sum_{n\geq 0} \Bigg( \sum_{l\geq 0} \binom{n+1}{l}x^l \sum_{j\geq 0}  (mj+r)^n  x^j \Bigg) \van  \notag \\
=&\sum_{k\idc}  \sum_{n\idc}\Bigg(\sum_{j=0}^{j=k}(-1)^j \binom{n+1}{j} \big(m(k-j)+r\big)^n\Bigg) \van x^k. \label{genmrr}
\end{align}
\eqref{amrnk} and \eqref{genmrr} imply 
\begin{equation}\label{expliamr}
\we=\sum_{j=0}^{j=k}(-1)^j \binom{n+1}{j} \big(m(k-j)+r\big)^n.
\end{equation}

\noindent
For $m \geq r$, the grammar of Hao $G$ can give us ideas about a structure that can interpret the numbers $\big(A_{m,r}(n,k)\big)_{n,k \geq 0}$. For example, let $m=3$ and $r=2$. The grammar is then
\begin{align*}
G = \{& u \to u v^3  \\ 
& v \to u^3 v \},
\end{align*}
and the associated differential operator is
\[
\opd{G}= u v^3 \frac{\partial}{\partial u} + u^3 v \frac{\partial}{\partial v}.
\]

\begin{center}
\[
\opd{G}^0uv^2 = uv^2
\]

\begin{tikzpicture}[scale=1,
	every node/.style={circle, draw, minimum size=4mm, inner sep=1pt, font=\small},
	box/.style={draw, rectangle, minimum width=4cm, minimum height=2cm}
	]

	\node[box, anchor=south west] (b1) at (0,0) {};
	
	\node[fill=white!70] (x3a) at (0.5,0.6) {$u$};
	\node[fill=blue!60] (x4a) at (1.9,0.6) {$v$};
	\node[fill=blue!60] (x5a) at (3.3,0.6) {$v$};
	
\end{tikzpicture}
\captionof{figure}{$\opd{G}^n(u^{m-r}v^r)$-structure for $n=0$}\label{bnamrzer}

\end{center}
\begin{center}
\[
\opd{G}uv^2 = uv^5 + 2u^4v^2
\]

\begin{tikzpicture}[scale=1,
	every node/.style={circle, draw, minimum size=4mm, inner sep=1pt, font=\small},
	box/.style={draw, rectangle, minimum width=4cm, minimum height=2cm}
	]
	
	\node[box, anchor=south west] (b1) at (0,0) {};
	\node[fill=darkgray] (1a) at (1.05,0.6) {1};
	\node[fill=white!70] (x1a) at (0.3,1.2) {$u$};
	\node[fill=blue!60] (x2a) at (0.8,1.2) {$v$};
	\node[fill=blue!60] (y1a) at (1.3,1.2) {$v$};
	\node[fill=blue!60] (y2a) at (1.8,1.2) {$v$};
	\draw (1a)--(x1a);
	\draw (1a)--(x2a);
	\draw (1a)--(y1a); \draw (1a)--(y2a);
	\node[fill=blue!60] (x4a) at (2.7,0.6) {$v$};
	\node[fill=blue!60] (x5a) at (3.5,0.6) {$v$};
	
	\node[box, anchor=south west] (b2) at (4.3,0) {};
	\node[fill=darkgray] (1b) at (6.25,0.6) {1};
	\node[fill=white!70] (x1b) at (5.5,1.2) {$u$};
	\node[fill=white!70] (y4b) at (6,1.2) {$u$};
	\node[fill=white!70] (y2b) at (6.5,1.2) {$u$};
	\node[fill=blue!60] (y3b) at (7,1.2) {$v$};
	\node[fill=white] (y1b) at (4.9,0.6) {$u$};
	\node[fill=blue!60] (x3b) at (7.6,0.6) {$v$};
	\draw (1b)--(x1b); \draw (1b)--(y4b); \draw (1b)--(y2b);  \draw (1b)--(y3b);
	
	\node[box, anchor=south west] (b3) at (8.6,0) {};
	\node[fill=blue!60] (x0c) at (10.2,0.6) {$v$};
	\node[fill=white!70] (x1c) at (10.75,1.2) {$u$};
	\node[fill=white!70] (x2c) at (11.25,1.2) {$u$};
	\node[fill=white!70] (x3c) at (11.75,1.2) {$u$};
	\node[fill=blue!60] (y2c) at (12.25,1.2) {$v$};
	\node[fill=white] (y1c) at (9,0.6) {$u$};
	\node[fill=darkgray] (1c) at (11.5,0.6) {$1$};
	\draw (1c)--(x1c); \draw (1c)--(x2c); \draw (1c)--(x3c);\draw (1c)--(y2c);

\end{tikzpicture}
\captionof{figure}{$\opd{G}^n(u^{m-r}v^r)$-structures for $n=1$}\label{bnamrun}
\end{center}

\begin{center}

\[
\opd{G}^2uv^2 = uv^8+13u^4v^5 + 4u^7v^2
\]

\begin{tikzpicture}[scale=1,
	every node/.style={circle, draw, minimum size=4mm, inner sep=1pt, font=\small},
	box/.style={draw, rectangle, minimum width=4.25cm, minimum height=2.4cm}
	]
	
	\node[box, anchor=south west] (b1) at (0,10) {};
	\node[fill=darkgray] (1a) at (1.8,10.6) {1};
	
	\node[fill=white!70] (xx1a) at (0.3,11.9) {$u$};
	\node[fill=blue!60] (xx2a) at (0.8,11.9) {$v$};
	\node[fill=blue!60] (xy1a) at (1.3,11.9) {$v$};
	\node[fill=blue!60] (xy2a) at (1.8,11.9) {$v$};
	
	\node[fill=darkgray!82] (x1a) at (1.05,11.2) {$2$};
	\node[fill=blue!60] (x2a) at (1.55,11.2) {$v$};
	\node[fill=blue!60] (y1a) at (2.05,11.2) {$v$};
	\node[fill=blue!60] (y2a) at (2.55,11.2) {$v$};
	\draw (1a)--(x1a);
	\draw (1a)--(x2a);
	\draw (1a)--(y1a); \draw (1a)--(y2a);
	\node[fill=blue!60] (x4a) at (3.25,10.6) {$v$};
	\node[fill=blue!60] (x5a) at (3.85,10.6) {$v$};
	
	\draw (x1a)--(xx1a);
	\draw (x1a)--(xx2a);
	\draw (x1a)--(xy1a); \draw (x1a)--(xy2a);

	\node[box, anchor=south west] (b2) at (4.4,10) {};
	\node[fill=darkgray] (11a) at (6.2,10.6) {1};
	
	\node[fill=white!70] (1xx1a) at (5.1,11.9) {$u$};
	\node[fill=white!70] (1xx2a) at (5.6,11.9) {$u$};
	\node[fill=white] (1xy1a) at (6.1,11.9) {$u$};
	\node[fill=blue!60] (1xy2a) at (6.6,11.9) {$v$};
	
	\node[fill=white!70] (1x1a) at (5.45,11.2) {$u$};
	\node[fill=darkgray!82] (1x2a) at (5.95,11.2) {$2$};
	\node[fill=blue!60] (1y1a) at (6.55,11.2) {$v$};
	\node[fill=blue!60] (1y2a) at (7.05,11.2) {$v$};
	\draw (11a)--(1x1a);
	\draw (11a)--(1x2a);
	\draw (11a)--(1y1a); \draw (11a)--(1y2a);
	\node[fill=blue!60] (1x4a) at (7.65,10.6) {$v$};
	\node[fill=blue!60] (1x5a) at (8.25,10.6) {$v$};
	
	\draw (1x2a)--(1xx1a);
	\draw (1x2a)--(1xx2a);
	\draw (1x2a)--(1xy1a); \draw (1x2a)--(1xy2a);

	\node[box, anchor=south west] (b3) at (8.8,10) {};
	\node[fill=darkgray] (11a) at (10.6,10.6) {1};
	
	\node[fill=white!70] (1xx1a) at (10,11.9) {$u$};
	\node[fill=white!70] (1xx2a) at (10.5,11.9) {$u$};
	\node[fill=white] (1xy1a) at (11,11.9) {$u$};
	\node[fill=blue!60] (1xy2a) at (11.5,11.9) {$v$};
	
	\node[fill=white!70] (1x1a) at (9.85,11.2) {$u$};
	\node[fill=blue!60] (1x2a) at (10.35,11.2) {$v$};
	\node[fill=darkgray!82] (1y1a) at (10.95,11.2) {$2$};
	\node[fill=blue!60] (1y2a) at (11.45,11.2) {$v$};
	\draw (11a)--(1x1a);
	\draw (11a)--(1x2a);
	\draw (11a)--(1y1a); \draw (11a)--(1y2a);
	\node[fill=blue!60] (1x4a) at (12.05,10.6) {$v$};
	\node[fill=blue!60] (1x5a) at (12.65,10.6) {$v$};
	
	\draw (1y1a)--(1xx1a);
	\draw (1y1a)--(1xx2a);
	\draw (1y1a)--(1xy1a); \draw (1y1a)--(1xy2a);

	\node[box, anchor=south west] (b4) at (13.2,10) {};
	\node[fill=darkgray] (11a) at (15,10.6) {1};
	
	\node[fill=white!70] (1xx1a) at (14.9,11.9) {$u$};
	\node[fill=white!70] (1xx2a) at (15.4,11.9) {$u$};
	\node[fill=white] (1xy1a) at (15.9,11.9) {$u$};
	\node[fill=blue!60] (1xy2a) at (16.4,11.9) {$v$};
	
	\node[fill=white!70] (1x1a) at (14.25,11.2) {$u$};
	\node[fill=blue!60] (1y2a) at (14.75,11.2) {$v$};
	\node[fill=blue!60] (1y1a) at (15.35,11.2) {$v$};
	\node[fill=darkgray!82] (1x2a) at (15.85,11.2) {$2$};
	\draw (11a)--(1x1a);
	\draw (11a)--(1x2a);
	\draw (11a)--(1y1a); \draw (11a)--(1y2a);
	\node[fill=blue!60] (1x4a) at (16.45,10.6) {$v$};
	\node[fill=blue!60] (1x5a) at (17.05,10.6) {$v$};
	
	\draw (1x2a)--(1xx1a);
	\draw (1x2a)--(1xx2a);
	\draw (1x2a)--(1xy1a); 
	\draw (1x2a)--(1xy2a);

	\node[box, anchor=south west] (c1) at (0,7.5) {};
	\node[fill=darkgray] (1a) at (0.93,8.1) {1};
	
	\node[fill=white!70] (xx1a) at (2.15,8.7) {$u$};
	\node[fill=white!60] (xx2a) at (2.6,8.7) {$u$};
	\node[fill=white!60] (xy1a) at (3.05,8.7) {$u$};
	\node[fill=blue!60] (xy2a) at (3.5,8.7) {$v$};
	
	\node[fill=white] (x4a) at (0.25,8.7) {$u$};
	\node[fill=blue!60] (x2a) at (0.7,8.7) {$v$};
	\node[fill=blue!60] (y1a) at (1.15,8.7) {$v$};
	\node[fill=blue!60] (y2a) at (1.6,8.7) {$v$};
	\draw (1a)--(x4a);
	\draw (1a)--(x2a);
	\draw (1a)--(y1a); \draw (1a)--(y2a);
	\node[fill=darkgray!82] (x1a) at (2.83,8.1) {$2$};
	\node[fill=blue!60] (x5a) at (3.85,8.1) {$v$};
	
	\draw (x1a)--(xx1a);
	\draw (x1a)--(xx2a);
	\draw (x1a)--(xy1a); \draw (x1a)--(xy2a);

	\node[box, anchor=south west] (c2) at (4.4,7.5) {};
	\node[fill=darkgray] (1a) at (5.33,8.1) {1};
	
	\node[fill=white!70] (xx1a) at (6.95,8.7) {$u$};
	\node[fill=white!60] (xx2a) at (7.4,8.7) {$u$};
	\node[fill=white!60] (xy1a) at (7.95,8.7) {$u$};
	\node[fill=blue!60] (xy2a) at (8.4,8.7) {$v$};
	
	\node[fill=white] (x4a) at (4.65,8.7) {$u$};
	\node[fill=blue!60] (x2a) at (5.1,8.7) {$v$};
	\node[fill=blue!60] (y1a) at (5.55,8.7) {$v$};
	\node[fill=blue!60] (y2a) at (6,8.7) {$v$};
	\draw (1a)--(x4a);
	\draw (1a)--(x2a);
	\draw (1a)--(y1a); \draw (1a)--(y2a);
	\node[fill=blue!60] (x5a) at (6.43,8.1) {$v$};
	\node[fill=darkgray!82] (x1a) at (7.63,8.1) {$2$};
	
	\draw (x1a)--(xx1a);
	\draw (x1a)--(xx2a);
	\draw (x1a)--(xy1a); \draw (x1a)--(xy2a);

	\node[box, anchor=south west] (c3) at (8.8,7.5) {};
	\node[fill=darkgray] (1b) at (11.45,8.1) {1};
	\node[fill=white!70] (x1b) at (10.9,8.7) {$u$};
	\node[fill=white!70] (y4b) at (11.35,8.7) {$u$};
	\node[fill=white!70] (y2b) at (11.8,8.7) {$u$};
	\node[fill=blue!60] (y3b) at (12.25,8.7) {$v$};
	\node[fill=darkgray!82] (x1a) at (9.7,8.1) {$2$};
	\node[fill=blue!60] (x3b) at (12.75,8.1) {$v$};
	\draw (1b)--(x1b);
	\draw (1b)--(y4b);
	\draw (1b)--(y2b);
	\draw (1b)--(y3b);
	\node[fill=white!70] (xx1a) at (9.05,8.7) {$u$};
	\node[fill=blue!60] (xx2a) at (9.5,8.7) {$v$};
	\node[fill=blue!60] (xy1a) at (9.95,8.7) {$v$};
	\node[fill=blue!60] (xy2a) at (10.4,8.7) {$v$};
	
	\draw (x1a)--(xx1a);
	\draw (x1a)--(xx2a);
	\draw (x1a)--(xy1a); \draw (x1a)--(xy2a);

	\node[box, anchor=south west] (c4) at (13.2,7.5) {};
	\node[fill=darkgray] (1b) at (15.15,8.1) {1};
	\node[fill=darkgray!82] (x1a) at (14.4,8.7) {$2$};
	\node[fill=white!70] (y4b) at (14.9,8.7) {$u$};
	\node[fill=white!70] (y2b) at (15.4,8.7) {$u$};
	\node[fill=blue!60] (y3b) at (15.9,8.7) {$v$};
	\node[fill=white] (y1b) at (13.7,8.1) {$u$};
	\node[fill=blue!60] (x3b) at (17,8.1) {$u$};
	\draw (1b)--(x1a);
	\draw (1b)--(y4b);
	\draw (1b)--(y2b);
	\draw (1b)--(y3b);
	\node[fill=white!70] (xx1a) at (13.6, 9.4) {$u$};
	\node[fill=blue!60] (xx2a) at (14.1, 9.4) {$v$};
	\node[fill=blue!60] (xy1a) at (14.6, 9.4) {$v$};
	\node[fill=blue!60] (xy2a) at (15.1, 9.4) {$v$};
	\draw (x1a)--(xx1a);
	\draw (x1a)--(xx2a);
	\draw (x1a)--(xy1a); \draw (x1a)--(xy2a);

	
	\node[box, anchor=south west] (c5) at (0,5) {};
	\node[fill=darkgray] (1b) at (1.95,5.6) {1};
	\node[fill=white!70] (y4b) at (1.2,6.2) {$u$};
	\node[fill=darkgray!82] (x1a) at (1.7,6.2) {$2$};
	\node[fill=white!70] (y2b) at (2.2,6.2) {$u$};
	\node[fill=blue!60] (y3b) at (2.7,6.2) {$v$};
	\node[fill=white] (y1b) at (0.5,5.6) {$u$};
	\node[fill=blue!60] (x3b) at (3.8,5.6) {$v$};
	\draw (1b)--(x1a);
	\draw (1b)--(y4b);
	\draw (1b)--(y2b);
	\draw (1b)--(y3b);
	\node[fill=white!70] (xx1a) at (0.8,6.9) {$u$};
	\node[fill=blue!60] (xx2a) at (1.3,6.9) {$v$};
	\node[fill=blue!60] (xy1a) at (1.8,6.9) {$v$};
	\node[fill=blue!60] (xy2a) at (2.3,6.9) {$v$};
	\draw (x1a)--(xx1a);
	\draw (x1a)--(xx2a);
	\draw (x1a)--(xy1a); 
	\draw (x1a)--(xy2a);

	\node[box, anchor=south west] (c6) at (4.4,5) {};
	\node[fill=darkgray] (1b) at (6.35,5.6) {1};
	\node[fill=white!70] (y2b) at (5.6,6.2) {$u$};
	\node[fill=white!70] (y4b) at (6.1,6.2) {$u$};
	\node[fill=darkgray!82] (x1a) at (6.6,6.2) {$2$};
	\node[fill=blue!60] (y3b) at (7.1,6.2) {$v$};
	\node[fill=white] (y1b) at (4.9,5.6) {$u$};
	\node[fill=blue!60] (x3b) at (8.2,5.6) {$v$};
	\draw (1b)--(x1a);
	\draw (1b)--(y4b);
	\draw (1b)--(y2b);
	\draw (1b)--(y3b);
	\node[fill=white!70] (xx1a) at (5.8,6.9) {$u$};
	\node[fill=blue!60] (xx2a) at (6.3,6.9) {$v$};
	\node[fill=blue!60] (xy1a) at (6.8,6.9) {$v$};
	\node[fill=blue!60] (xy2a) at (7.3,6.9) {$v$};
	\draw (x1a)--(xx1a);
	\draw (x1a)--(xx2a);
	\draw (x1a)--(xy1a); 
	\draw (x1a)--(xy2a);

	\node[box, anchor=south west] (c7) at (8.8,5) {};
	\node[fill=darkgray] (1b) at (10.75,5.6) {1};
	\node[fill=white!70] (y3b) at (10,6.2) {$u$};
	\node[fill=white!70] (y4b) at (10.5,6.2) {$u$};
	\node[fill=white!70] (y2b) at (11,6.2) {$u$};
	\node[fill=darkgray!82] (x1a) at (11.5,6.2) {$2$};
	\node[fill=white] (y1b) at (9.3,5.6) {$u$};
	\node[fill=blue!60] (x3b) at (12.6,5.6) {$v$};
	\draw (1b)--(x1a);
	\draw (1b)--(y4b);
	\draw (1b)--(y2b);
	\draw (1b)--(y3b);
	\node[fill=white!70] (xx1a) at (10.6,6.9) {$u$};
	\node[fill=white!70] (xx2a) at (11.1,6.9) {$u$};
	\node[fill=white] (xy1a) at (11.6,6.9) {$u$};
	\node[fill=blue!60] (xy2a) at (12.1,6.9) {$v$};
	\draw (x1a)--(xx1a);
	\draw (x1a)--(xx2a);
	\draw (x1a)--(xy1a); 
	\draw (x1a)--(xy2a);

	\node[box, anchor=south west] (c8) at (13.2,5.0) {};
	\node[fill=darkgray] (1b) at (14.5,5.6) {1};
	\node[fill=white!70] (x1b) at (13.65,6.2) {$u$};
	\node[fill=white!70] (y4b) at (14.1,6.2) {$u$};
	\node[fill=white!70] (y2b) at (14.55,6.2) {$u$};
	\node[fill=blue!60] (y3b) at (15,6.2) {$v$};
	\node[fill=white] (y1b) at (13.5,5.6) {$u$};
	\node[fill=darkgray!82] (31b) at (16.5,5.6) {$2$};
	\draw (1b)--(x1b);
	\draw (1b)--(y4b);
	\draw (1b)--(y2b);
	\draw (1b)--(y3b);
	
	\node[fill=white!70] (2x1b) at (15.75,6.2) {$u$};
	\node[fill=white!70] (2y4b) at (16.2,6.2) {$u$};
	\node[fill=white!70] (2y2b) at (16.7,6.2) {$u$};
	\node[fill=blue!60] (2y3b) at (17.2,6.2) {$v$};
	\draw (31b)--(2x1b);
	\draw (31b)--(2y4b);
	\draw (31b)--(2y2b);
	\draw (31b)--(2y3b);

	
	\node[box, anchor=south west] (d1) at (0,2.5) {};
	\node[fill=darkgray!82] (x0c) at (1,3.1) {$2$};
	\node[fill=white!70] (x1c) at (2.55,3.7) {$u$};
	\node[fill=white!70] (x2c) at (3,3.7) {$u$};
	\node[fill=white!70] (x3c) at (3.45,3.7) {$u$};
	\node[fill=blue!60] (y2c) at (3.95,3.7) {$v$};
	\node[fill=blue!60] (y1c) at (2.15,3.1) {$v$};
	\node[fill=darkgray] (1c) at (3.15,3.1) {$1$};
	\draw (1c)--(x1c); 
	\draw (1c)--(x2c); 
	\draw (1c)--(x3c);
	\draw (1c)--(y2c);
	
	\node[fill=white!70] (xx1c) at (0.25,3.7) {$u$};
	\node[fill=blue!60] (xx2c) at (0.75,3.7) {$v$};
	\node[fill=blue!60] (xx3c) at (1.2,3.7) {$v$};
	\node[fill=blue!60] (xy2c) at (1.65,3.7) {$v$};
	\draw (x0c)--(xx1c); 
	\draw (x0c)--(xx2c); 
	\draw (x0c)--(xx3c);
	\draw (x0c)--(xy2c);

	\node[box, anchor=south west] (d2) at (4.4,2.5) {};
	\node[fill=darkgray!82] (x0c) at (5.7,3.1) {$2$};
	\node[fill=white!70] (x1c) at (6.95,3.7) {$u$};
	\node[fill=white!70] (x2c) at (7.4,3.7) {$u$};
	\node[fill=white!70] (x3c) at (7.85,3.7) {$u$};
	\node[fill=blue!60] (y2c) at (8.35,3.7) {$v$};
	\node[fill=white] (y1c) at (4.65,3.1) {$u$};
	\node[fill=darkgray] (1c) at (7.55,3.1) {$1$};
	\draw (1c)--(x1c); 
	\draw (1c)--(x2c); 
	\draw (1c)--(x3c);
	\draw (1c)--(y2c);
	
	\node[fill=white!70] (xx1c) at (4.8,3.7) {$u$};
	\node[fill=white!70] (xx2c) at (5.35,3.7) {$u$};
	\node[fill=white!70] (xx3c) at (5.8,3.7) {$u$};
	\node[fill=blue!60] (xy2c) at (6.35,3.7) {$v$};
	\draw (x0c)--(xx1c); 
	\draw (x0c)--(xx2c); 
	\draw (x0c)--(xx3c);
	\draw (x0c)--(xy2c);

	\node[box, anchor=south west] (d3) at (8.8,2.5) {};
	\node[fill=blue!60] (x0c) at (10.1,3.1) {$v$};
	\node[fill=darkgray!82] (x1a) at (11.25,3.7) {$2$};
	\node[fill=white!70] (x2c) at (11.75,3.7) {$u$};
	\node[fill=white!70] (x3c) at (12.25,3.7) {$u$};
	\node[fill=blue!60] (y2c) at (12.75,3.7) {$v$};
	\node[fill=white] (y1c) at (9.2,3.1) {$u$};
	\node[fill=darkgray] (1c) at (12.0,3.1) {$1$};
	\draw (1c)--(x1a); 
	\draw (1c)--(x2c); 
	\draw (1c)--(x3c);
	\draw (1c)--(y2c);
	
	\node[fill=white!70] (xx1a) at (10.6,4.4) {$u$};
	\node[fill=blue!60] (xx2a) at (11.1,4.4) {$v$};
	\node[fill=blue!60] (xy1a) at (11.6,4.4) {$v$};
	\node[fill=blue!60] (xy2a) at (12.1,4.4) {$v$};
	\draw (x1a)--(xx1a);
	\draw (x1a)--(xx2a);
	\draw (x1a)--(xy1a); 
	\draw (x1a)--(xy2a);

	\node[box, anchor=south west] (d4) at (13.2,2.5) {};
	\node[fill=blue!60] (x0c) at (14.5,3.1) {$v$};
	\node[fill=darkgray!82] (x1a) at (16.15,3.7) {$2$};
	\node[fill=white!70] (x2c) at (15.65,3.7) {$u$};
	\node[fill=white!70] (x3c) at (16.65,3.7) {$u$};
	\node[fill=blue!60] (y2c) at (17.15,3.7) {$v$};
	\node[fill=white] (y1c) at (13.6,3.1) {$u$};
	\node[fill=darkgray] (1c) at (16.4,3.1) {$1$};
	\draw (1c)--(x1a); 
	\draw (1c)--(x2c); 
	\draw (1c)--(x3c);
	\draw (1c)--(y2c);
	
	\node[fill=white!70] (xx1a) at (15.2,4.4) {$u$};
	\node[fill=blue!60] (xx2a) at (15.7,4.4) {$v$};
	\node[fill=blue!60] (xy1a) at (16.2,4.4) {$v$};
	\node[fill=blue!60] (xy2a) at (16.7,4.4) {$v$};
	\draw (x1a)--(xx1a);
	\draw (x1a)--(xx2a);
	\draw (x1a)--(xy1a); 
	\draw (x1a)--(xy2a);
	
	\node[box, anchor=south west] (e1) at (4.4,0) {};
	\node[fill=blue!60] (x0c) at (5.7,0.6) {$v$};
	\node[fill=darkgray!82] (x1a) at  (7.85,1.2) {$2$};
	\node[fill=white!70] (x2c) at (7.35,1.2) {$u$};
	\node[fill=white!70] (x3c) at (6.85,1.2) {$u$};
	\node[fill=blue!60] (y2c) at (8.35,1.2) {$v$};
	\node[fill=white] (y1c) at (4.8,0.6) {$u$};
	\node[fill=darkgray] (1c) at (7.6,0.6) {$1$};
	\draw (1c)--(x1a); 
	\draw (1c)--(x2c); 
	\draw (1c)--(x3c);
	\draw (1c)--(y2c);
	
	\node[fill=white!70] (xx1a) at (6.9,1.9) {$u$};
	\node[fill=blue!60] (xx2a) at (7.4,1.9) {$v$};
	\node[fill=blue!60] (xy1a) at (7.9,1.9) {$v$};
	\node[fill=blue!60] (xy2a) at (8.4,1.9) {$v$};
	\draw (x1a)--(xx1a);
	\draw (x1a)--(xx2a);
	\draw (x1a)--(xy1a); 
	\draw (x1a)--(xy2a);

	\node[box, anchor=south west] (e2) at (8.8,0) {};
	\node[fill=blue!60] (x0c) at (10.1,0.6) {$v$};
	\node[fill=darkgray!82] (x1a) at (12.15,1.2) {$2$};
	\node[fill=white!70] (x2c) at (11.15,1.2) {$u$};
	\node[fill=white!70] (x3c) at (10.65,1.2) {$u$};
	\node[fill=white!70] (y2c) at  (11.65,1.2) {$u$};
	\node[fill=white] (y1c) at (9.2,0.6) {$u$};
	\node[fill=darkgray] (1c) at (11.4,0.6) {$1$};
	\draw (1c)--(x1a); 
	\draw (1c)--(x2c); 
	\draw (1c)--(x3c);
	\draw (1c)--(y2c);
	
	\node[fill=white!70] (xx1a) at (11.3,1.9) {$u$};
	\node[fill=white!70] (xx2a) at (11.8,1.9) {$u$};
	\node[fill=white] (xy1a) at (12.3,1.9) {$u$};
	\node[fill=blue!60] (xy2a) at (12.8,1.9) {$v$};
	\draw (x1a)--(xx1a);
	\draw (x1a)--(xx2a);
	\draw (x1a)--(xy1a); 
	\draw (x1a)--(xy2a);
	
\end{tikzpicture}
\captionof{figure}{$\opd{G}^n(u^{m-r}v^r)$-structures for $n=2$}\label{bnamrd}
\end{center}
\noindent
From Proposition \ref{phaores}, $A_{m,r}(n,k)$ is the number of $\opd{G}^n(u^{m-r}v^r)$-structures having $km+r$ $v$-leaves (blue in our Figures \ref{bnamrzer}, \ref{bnamrun}, \ref{bnamrd}).
Moreover, we notice that a $\opd{G}^n(u^{m-r}v^r)$-structure has $f_n = m(n+1)$ leaves, and if $a_n$ denotes the total number of $\opd{G}^n(u^{m-r}v^r)$-structures, then
$a_0 = 1$ and for $n \geq 1$, $a_n = a_{n-1} \times f_{n-1} = n m a_{n-1}$. Therefore,
\[
a_n = n! \, m^n,
\]
which reaffirms relation \eqref{sumamr}.

\begin{example}\label{witn2esp}
Consider the sequence $\big( F(n,k)\big)_{n,k\ge 0}$ satisfying \eqref{recu}, with $b_{n,k}=1$ and $a_{n,k}=a_0+a_1 k$, $a_1\neq 0$. The system of differential equations associated with the grammar $G$ in \eqref{gram2} is therefore
\begin{align*}
	\left\{
	\begin{array}{cl}
		U' &= U V^{a_1}, \quad U(0) = u,  \\
		V' &= V, \quad\quad\quad V(0) = v.
	\end{array}
	\right.
\end{align*}
\end{example}
\noindent
The analytic solution is
\begin{align}\label{sol22}
\left\{
\begin{array}{cl}
	U(t) &= u \exp{\Bigg(\frac{v^{a_1}}{a_1}\big(e^{a_1 t}-1\big)\Bigg)},  \\
	V(t) &= v e^t.
\end{array}
\right.
\end{align}
\noindent
Recall that the Bell polynomial is defined by
\[
B_n(\lambda) = \sum_{k=0}^{n} S(n,k) \lambda^k.
\]
The classical fact concerning the Stirling numbers then gives
\begin{equation}\label{bels}
\sum_{n\ge 0} B_n(\lambda) \frac{t^n}{n!} = \exp{\Big(\lambda (e^t - 1)\Big)}.
\end{equation}
\noindent
Equations \eqref{sol22} and \eqref{bels} imply
\begin{equation}\label{eqsol2}
U(t) \big(V(t)\big)^{a_0} = u v^{a_0} \sum_{n \ge 0} \left( \sum_{k=0}^{n} \binom{n}{k} a_1^k a_0^{\, n-k} B_k \Big( \frac{v^{a_1}}{a_1} \Big) \right) \frac{t^n}{n!}.
\end{equation}
\noindent
On the other hand, Proposition \ref{phaores} states that
\begin{equation}\label{eqphao}
\opd{G}^n(uv^{a_0}) = \sum_{k=0}^{n} F(n,k) \, u v^{a_1 k + a_0}.
\end{equation}
\noindent
Now, combining these facts with equation \eqref{algG} and Proposition \ref{chendumrabez}, we obtain
\begin{equation}\label{touchardtnk}
\sum_{k\ge 0} F(n,k) \alpha^k = \sum_{k\ge 0} \binom{n}{k} a_1^k a_0^{\, n-k} B_k\Big( \frac{\alpha}{a_1} \Big),
\end{equation}
by substituting $\alpha = v^{a_1}$.\newline
\noindent
Returning to the expression for $B_k(x)$, we have
\begin{equation}\label{tnka2zero}
F(n,k) = \sum_{j=k}^{n} \binom{n}{j} a_0^{\, n-j} a_1^{\, j-k} S(j,k).
\end{equation}
\noindent
Next, consider $\opd{G}^n(uv^{a_0})$, with
\[
\opd{G} = u v^{a_1} \frac{\partial}{\partial u} + v \frac{\partial}{\partial v}.
\]
We now look at the example where $a_0 = 2$ and $a_1 = 2$.

\begin{center}
\[
\opd{G}^0uv^2 = uv^2
\]
\begin{tikzpicture}[scale=1,
	every node/.style={circle, draw, minimum size=4mm, inner sep=1pt, font=\small},
	box/.style={draw, rectangle, minimum width=4cm, minimum height=2cm}
	]

	\node[box, anchor=south west] (b1) at (0,-5) {};         
	
	\node[fill=white!70] (x3a) at (0.5,-4.4) {$u$};
	\node[fill=blue!60] (x4a) at (1.9,-4.4) {$v$};
	\node[fill=blue!60] (x5a) at (3.3,-4.4) {$v$};
	
\end{tikzpicture}
\captionof{figure}{$\opd{G}^n(uv^{a_0})$-structure for $n=0$}\label{ca1str0}

\end{center}
\begin{center}
\[
\opd{G}uv^2 = uv^4 + 2uv^2
\]

\begin{tikzpicture}[scale=1,
	every node/.style={circle, draw, minimum size=4mm, inner sep=1pt, font=\small},
	box/.style={draw, rectangle, minimum width=4cm, minimum height=2cm}
	]
	
	\node[box, anchor=south west] (b1) at (0,-5) {};
	\node[fill=darkgray] (1a) at (1.05,-4.4) {1};
	\node[fill=white!70] (x1a) at (0.3,-3.8) {$u$};
	\node[fill=blue!60] (y1a) at (1.05,-3.8) {$v$};
	\node[fill=blue!60] (y2a) at (1.8,-3.8) {$v$};
	\draw (1a)--(x1a);
	\draw (1a)--(y1a); 
	\draw (1a)--(y2a);
	\node[fill=blue!60] (x4a) at (2.7,-4.4) {$v$};
	\node[fill=blue!60] (x5a) at (3.5,-4.4) {$v$};
	
	\node[box, anchor=south west] (b2) at (4.3,-5) {};
	\node[fill=darkgray] (1b) at (6.25,-4.4) {1};
	\node[fill=blue!60] (y2b) at (6.25,-3.8) {$v$};
	\node[fill=white] (y1b) at (4.9,-4.4) {$u$};
	\node[fill=blue!60] (x3b) at (7.6,-4.4) {$v$};
	\draw (1b)--(y2b);

	\node[box, anchor=south west] (b3) at (8.6,-5) {};
	\node[fill=blue!60] (x0c) at (10.2,-4.4) {$v$};
	\node[fill=blue!60] (x2c) at (11.5,-3.8) {$v$};
	\node[fill=white] (y1c) at (9,-4.4) {$u$};
	\node[fill=darkgray] (1c) at (11.5,-4.4) {$1$};
	\draw (1c)--(x2c); 
\end{tikzpicture}

\captionof{figure}{$\opd{G}^n(uv^{a_0})$-structures for $n=1$}\label{ca1str1}
\end{center}

\begin{center}
\[
\opd{G}^2uv^2 = uv^6+4uv^4 + 2uv^4 + 4uv^2
\]

\begin{tikzpicture}[scale=1,
	every node/.style={circle, draw, minimum size=4mm, inner sep=1pt, font=\small},
	box/.style={draw, rectangle, minimum width=4.25cm, minimum height=2.4cm}
	]
	
	\node[box, anchor=south west] (b1) at (0,5) {};
	\node[fill=darkgray] (1a) at (1.8,5.6) {1};
	
	\node[fill=white!70] (xx1a) at (0.3,6.9) {$u$};
	\node[fill=blue!60] (xx2a) at (1.05,6.9) {$v$};
	\node[fill=blue!60] (xy2a) at (1.8,6.9) {$v$};
	\node[fill=darkgray!82] (x1a) at (1.05,6.2) {$2$};
	\node[fill=blue!60] (x2a) at (1.8,6.2) {$v$};
	\node[fill=blue!60] (y2a) at (2.55,6.2) {$v$};
	\draw (1a)--(x1a);
	\draw (1a)--(x2a);
	\draw (1a)--(y2a);
	\node[fill=blue!60] (x4a) at (3.25,5.6) {$v$};
	\node[fill=blue!60] (x5a) at (3.85,5.6) {$v$};
	
	\draw (x1a)--(xx1a);
	\draw (x1a)--(xx2a); 
	\draw (x1a)--(xy2a);
	
	\node[box, anchor=south west] (b2) at (4.4,5) {};
	\node[fill=darkgray] (11a) at (6.2,5.6) {1};
	\node[fill=blue!60] (1xx2a) at (6.2,6.9) {$v$};

	\node[fill=white!70] (1x1a) at (5.45,6.2) {$u$};
	\node[fill=darkgray!82] (1x2a) at (6.2,6.2) {$2$};
	
	\node[fill=blue!60] (1y2a) at (7.05,6.2) {$v$};
	\draw (11a)--(1x1a);
	\draw (11a)--(1x2a);
	
	\draw (11a)--(1y2a);
	\node[fill=blue!60] (1x4a) at (7.65,5.6) {$v$};
	\node[fill=blue!60] (1x5a) at (8.25,5.6) {$v$};

	\draw (1x2a)--(1xx2a);

	\node[box, anchor=south west] (b3) at (8.8,5) {};
	\node[fill=darkgray] (11a) at (10.6,5.6) {1};
	
	\node[fill=blue!60] (1xx1a) at (11.45,6.9) {$v$};

	\node[fill=white!70] (1x1a) at (9.85,6.2) {$u$};
	\node[fill=blue!60] (1x2a) at (10.6,6.2) {$v$};
	\node[fill=darkgray!82] (1y1a) at (11.45,6.2) {$2$};
	
	\draw (11a)--(1x1a);
	\draw (11a)--(1x2a);
	\draw (11a)--(1y1a);
	
	\node[fill=blue!60] (1x4a) at (12.05,5.6) {$v$};
	\node[fill=blue!60] (1x5a) at (12.65,5.6) {$v$};
	\draw (1y1a)--(1xx1a);

	\node[box, anchor=south west] (b4) at (13.2,5) {};
	\node[fill=darkgray] (11a) at (15,5.6) {1};
	
	\node[fill=blue!60] (1xx1a) at (16.45,6.2) {$v$};
	
	\node[fill=white!70] (1x1a) at (14.25,6.2) {$u$};
	\node[fill=blue!60] (1y2a) at (14.75,6.2) {$v$};
	\node[fill=blue!60] (1y1a) at (15.5,6.2) {$v$};
	\draw (11a)--(1x1a);
	\draw (11a)--(1y1a); 
	\draw (11a)--(1y2a);
	\node[fill=darkgray!82] (1x4a) at (16.45,5.6) {$2$};
	\node[fill=blue!60] (1x5a) at (17.05,5.6) {$v$};
	
	\draw (1x4a)--(1xx1a);

	
	\node[box, anchor=south west] (c1) at (0,2.5) {};
	\node[fill=darkgray] (1a) at (0.93,3.1) {1};
	\node[fill=blue!60] (xx2a) at (3.85,3.7) {$v$};

	\node[fill=white] (x4a) at (0.25,3.7) {$u$};
	\node[fill=blue!60] (x2a) at (0.93,3.7) {$v$};
	\node[fill=blue!60] (y2a) at (1.6,3.7) {$v$};
	\draw (1a)--(x4a);
	\draw (1a)--(x2a);
	\draw (1a)--(y2a);
	\node[fill=darkgray!82] (x1a) at (3.85,3.1) {$2$};
	\node[fill=blue!60] (x5a) at (2.83,3.1) {$v$};

	\draw (x1a)--(xx2a);

	\node[box, anchor=south west] (c2) at (4.4,2.5) {};
	\node[fill=darkgray!82] (1a) at (5.33,3.1) {2};
	
	\node[fill=blue!60] (xx1a) at (6.75,3.7){$v$};
	\node[fill=white] (x4a) at (4.65,3.7) {$u$};
	\node[fill=blue!60] (y1a) at (5.33,3.7) {$v$};
	\node[fill=blue!60] (y2a) at (6,3.7) {$v$};
	\draw (1a)--(x4a);
	
	\draw (1a)--(y1a); \draw (1a)--(y2a);
	\node[fill=darkgray] (x5a) at (6.75,3.1) {$1$};
	\node[fill=blue!60] (x1a) at (7.63,3.1) {$v$};
	
	\draw (x5a)--(xx1a);
	
	
	\node[box, anchor=south west] (c3) at (8.8,2.5) {};
	\node[fill=darkgray] (1b) at (11.45,3.1) {1};
	\node[fill=blue!60] (x1b) at (11.45,4.4) {$v$};
	\node[fill=darkgray!82] (y3b) at (11.45,3.7) {$2$};
	\node[fill=white!70] (x1a) at (9.7,3.1) {$u$};
	\node[fill=blue!60] (x3b) at (12.75,3.1) {$v$};
	\draw (y3b)--(x1b);
	\draw (1b)--(y3b);

	
	\node[box, anchor=south west] (c4) at (13.2,2.5) {};
	\node[fill=darkgray] (1b) at (15.15,3.1) {1};
	\node[fill=blue!60] (y4b) at (15.15,3.7) {$v$};
	\node[fill=blue!60] (y2b) at (17,3.7) {$v$};
	\node[fill=white] (y1b) at (13.7,3.1) {$u$};
	\node[fill=darkgray!82] (x3b) at (17,3.1) {$2$};
	\draw (x3b)--(y2b);
	\draw (1b)--(y4b);

	\node[box, anchor=south west] (c5) at (0,0) {};
	\node[fill=blue!60] (1b) at (2.75,0.6) {v};
	
	\node[fill=blue!60] (x1a) at (1.8,1.2) {$v$};
	\node[fill=white!70] (y2b) at (0.25,1.2) {$u$};
	\node[fill=blue!60] (y3b) at (1,1.2) {$v$};
	\node[fill=darkgray!82] (y1b) at (1,0.6) {$2$};
	\node[fill=darkgray] (x3b) at (3.8,0.6) {$1$};
	\draw (y1b)--(x1a);
	\draw (y1b)--(y3b);
	\draw (y1b)--(y2b);
	
	\node[fill=blue!60] (xx2a) at (3.8,1.2) {$v$};
	\draw (x3b)--(xx2a);

	
	\node[box, anchor=south west] (c6) at (4.4,0) {};
	\node[fill=darkgray!80] (1b) at (6.35,0.6) {2};
	\node[fill=blue!60] (y2b) at (6.35,1.2) {$v$};
	\node[fill=blue!60] (y4b) at (8.2,1.2) {$v$};
	\node[fill=white] (y1b) at (4.9,0.6) {$u$};
	\node[fill=darkgray] (x3b) at (8.2,0.6) {$1$};
	
	\draw (x3b)--(y4b);
	\draw (1b)--(y2b);

	
	\node[box, anchor=south west] (c7) at (8.8,0) {};
	\node[fill=blue!60] (1b) at (10.75,0.6) {v};
	\node[fill=darkgray!82] (y3b) at (12.6,1.2) {$2$};
	\node[fill=white] (y1b) at (9.3,0.6) {$u$};
	\node[fill=darkgray] (x3b) at (12.6,0.6) {$1$};
	\node[fill=blue!60] (xx1a) at (12.6,1.9) {$v$};
	\draw (x3b)--(y3b); \draw (xx1a)--(y3b);


\end{tikzpicture}
\captionof{figure}{$\opd{G}^n(uv^{a_0})$-structures for $n=2$}  \label{ca1str2}    
\end{center}
\noindent
From equation \eqref{eqphao}, $F(n,k)$ counts the number of $\opd{G}^n(uv^{a_0})$-structures having $a_1 k + a_0$ $v$-leaves (blue in our Figures \ref{ca1str0}, \ref{ca1str1}, \ref{ca1str2}).

\begin{example}\label{witn1esp}
Consider the sequence $\big( F(n,k) \big)_{n,k \ge 0}$ satisfying \eqref{recu}, with $b_{n,k}=1$ and $a_{n,k}=a_0 + a_2 n$, $a_2 \neq 0$. The system of differential equations associated with the grammar $G_2$ in \eqref{gram2} is therefore
\begin{align*}
	\left\{
	\begin{array}{cl}
		U' &= U V^{a_2}, \quad U(0) = u,  \\
		V' &= V^{a_2+1}, \quad V(0) = v.
	\end{array}
	\right.
\end{align*}
\end{example} \noindent
The analytic solution is
\begin{align}\label{sol233}
\left\{
\begin{array}{cl}
	U(t) &= \frac{u}{v} \, (v^{-a_2} - a_2 t)^{-\frac{1}{a_2}},  \\
	V(t) &= (v^{-a_2} - a_2 t)^{-\frac{1}{a_2}}.
\end{array}
\right.
\end{align}
\noindent
Thus, we have
\begin{equation*}
U(t) \big(V(t)\big)^{a_0 + a_2} = \frac{u}{v} \, (v^{-a_2} - a_2 t)^{-\alpha} 
= u v^{a_0 + a_2} \sum_{n \ge 0} \agenm{\alpha}{n} a_2^n v^{a_2 n} \frac{t^n}{n!},
\end{equation*}
where $\alpha = \frac{1 + a_0 + a_2}{a_2}$.
\newline
Again, as before, using Proposition \ref{phaores}, we deduce
\begin{equation*}
\sum_{n \ge 0} \sum_{k=0}^{n} F(n,k) \, u v^{a_0 + a_2 + a_2 n} \frac{t^n}{n!} 
= u v^{a_0 + a_2} \sum_{n \ge 0} \agenm{\alpha}{n} a_2^n v^{a_2 n} \frac{t^n}{n!}.
\end{equation*}

This equation implies that
\begin{equation*}
\sum_{k=0}^{n} F(n,k) = \agenm{\alpha}{n} a_2^n = \left( \frac{1 + a_0 + a_2}{a_2} \right)^{(\underline{n})} a_2^n.
\end{equation*}

\begin{remark}
Briefly, Proposition \ref{phaores} states that
\begin{align*}
	\opd{G}^n(uv^{a_0 + a_2}) = \sum_k T(n,k) \, u v^{a_2 n + a_2 + a_0},
\end{align*}
which does not allow us to distinguish the individual $T(n,k)$ according to the number of $v$-leaves; all $\opd{G}^n(uv^{a_0 + a_2})$-structures have the same number of leaves: one $u$-leaf and $a_2 n + a_2 + a_0$ $v$-leaves.
\end{remark}
\noindent
Therefore, it is necessary to introduce additional properties to determine the individual $T(n,k)$. We now briefly discuss a certain type of grammar: the type $(E)$ grammar, introduced in Randrianirina \cite{randrianirinahal-05091963}.

\begin{definition}[\cite{randrianirinahal-05091963}]
A type $(E)$ grammar is a grammar $G$ on $\{z\} \cup X$, with $z \notin X$, defined by
\begin{align*}
	G = \{ & z \to z g(\vec{x})  \\
	& x_i \to h_i(\vec{x}) : \quad x_i \in X \},
\end{align*}
where $g(\vec{x})$ is a polynomial in $C[X]$.
\end{definition}
\noindent
We associate to this grammar a system of differential equations. By solving it combinatorially, the $\mathcal{Z}$-structure is a species of structures as illustrated in Figure \ref{ZG-structure}.
\begin{figure}[ht]
\centering
\begin{tikzpicture}[scale=1, every node/.style={font=\small}]
	\tikzset{
		spine/.style={blue, thick},
		ptspine/.style={circle, fill=blue, inner sep=2.5pt},
		gbranch/.style={gray, thick, rounded corners=2pt},
		gslabel/.style={gray, font=\small},
		suspension/.style={fill=black, inner sep=0.9pt, circle}
	}
	
	\coordinate (P1) at (0,-0.3);
	\coordinate (P2) at (-0.9,0.8);
	\coordinate (P3) at (-1.8,1.9);
	\coordinate (Pk) at (-3,3.3667);
	
	\node[ptspine, label=left:{\textcolor{blue}{$i_1$=1}}] at (P1) (A1) {};
	\node[ptspine, label=left:{\textcolor{blue}{$i_2$}}] at (P2) (A2) {};
	\node[ptspine, label=left:{\textcolor{blue}{$i_3$}}] at (P3) (A3) {};
	\node[ptspine, label=left:{\textcolor{blue}{$i_k$}}] at (Pk) (Ak) {};
	
	\draw[spine] (A1)--(A2)--(A3);
	\foreach \t in {0.4,0.55,0.7} {
		\fill[blue] ($(A3)!\t!(Ak)$) circle (1.3pt);
	}
	\draw[spine] (Ak)--($(A3)!0.95!(Ak)$);
	
	\foreach \pt/\name in {A1/1, A2/2, A3/3, Ak}{
		\draw[gbranch] (\pt) .. controls ($( \pt )+(0.99,0.8)$) .. ($( \pt )+(1.3,1.3)$);
		\draw[gbranch] (\pt) .. controls ($( \pt )+(1.2,0.3)$) .. ($( \pt )+(1.8,0.50)$)
		node[right, gslabel] {$\mathcal{G}-$structure};
		\draw[very thick] ($( \pt )+(1.5,0.2)$) .. controls ($( \pt )+(1.4,0.55)$) .. ($( \pt )+(0.85,0.95)$);
		\coordinate (s\name) at ($( \pt )+(1,0.4)$);
		\foreach \i in {0,1,2}{
			\fill[suspension] ($(s\name)+(-0.07*\i,0.09*\i)$) circle (1pt);
		}
	}
	
	\node[circle, draw=black, fill=white, minimum size=9pt, label=center:{$z$}] 
	at ($(Ak)+(-0.7,0.4)$) (leafz) {};
	\draw[thick] (Ak) -- (leafz);

	\draw[rounded corners=6pt, thick] (-4.5,-0.8) rectangle (3.8,5.3);
	
\end{tikzpicture}
\caption{$\mathcal{Z}$-structure in terms of $\mathcal{G}$-structure.
}
\label{ZG-structure}
\end{figure}
\noindent Hence, $\mathcal{Z}$ is the set of $\mathcal{G}$-structures, with
\[
\mathcal{G} = \int g\big(\vec{\mathcal{X}}\big).
\]
\noindent
Now consider the blue line, which we call the backbone. For the language in Randrianirina \cite{randrianirinahal-05091963}, each point on the backbone is associated with a connected component, and the author associates the species $\mathcal{Z}$ to a sequence $\big( Z(n,k) \big)_{n,k \ge 0}$, where $Z(n,k)$ is the number of $\mathcal{Z}$-structures on $n$ points having $k$ connected components.
\\ \noindent
Now consider a sequence $\big( F(n,k) \big)_{n,k \ge 0}$ satisfying \eqref{recu} with $\coul{a} = a_2 n + a_1 k + a_0$, $\coul{b} = 1$, and the usual initial conditions. The grammar $G$ from \eqref{gram2} associated with this sequence is thus
\begin{align*}
G = \{ & u \rightarrow u v^{a_1 + a_2}; \\
& v \rightarrow v^{a_2 + 1}. \} 
\end{align*}
\noindent
This grammar is of type $(E)$, and we have the following important remark. 
\begin{prop}
Every $\opd{G}^n(uv^{a_0 + a_2})$-structure $s$ is canonically decomposed as $s = (s_1, s_2)$, where $s_1$ is a $\mathcal{U}$-structure and $s_2$ is a $\mathcal{V}^{a_0 + a_2}$-structure. Therefore, $F(n,k)$ is the number of $\opd{G}_2^n(uv^{a_0 + a_2})$-structures $s$ such that $s_1$ has $k$ connected components.
\end{prop}

\begin{proof}
First, Proposition \ref{phaores} ensures that $\suit{F}$ is the number of 
$\opd{G}^n(uv^{a_0+a_2})$-structures having $a_2 n + a_1 k + a_0 + a_2$ $v$-leaves. 
Now, let $G(n,k)$ denote the number of $\opd{G}^n(uv^{a_0+a_2})$-structures $s$ whose $s_1$ has $k$ connected components.
\\ \noindent
We proceed by induction. For $n=0$, the unique structure has $a_0 + a_2$ $v$-leaves, 
so the relation holds trivially. For $n \ge 1$, assume the relation holds up to $n-1$, 
i.e., $\suit{T} = \suit{G}$, and show it for $n$. Consider a $\opd{G}_2^n(uv^{a_0+a_2})$-structure $s$ with $k$ connected components. 
We focus on the position of point $n$, with two possible cases:
\begin{itemize}
	\item Case 1: point $n$ is on the spine (at $i_k$). This occurs if and only if 
	the last derivation was applied to the unique $u$-leaf. This last step increases the number of connected components.
	\item Case 2: point $n$ is not on the spine. Then the last derivation was applied to one of the $v$-leaves. 
	This last step does not increase the number of connected components. Using the induction hypothesis 
	and Proposition \ref{phaores}, there are $a_2 (n-1) + a_1 k + a_0$ choices for this last step.
\end{itemize}
This correspondence ensures that $\big(\suit{G}\big)\nkid$ satisfies
\[
\suit{G} = G(n-1,k-1) + (a_2 n + a_1 k + a_0) G(n-1,k),
\]
which completes the proof.
\end{proof}

\begin{corollary}
Let $(\suit{T})$ be a sequence satisfying the usual initial conditions with
\[
\suit{T} = T(n-1,k-1) + (a_2 n + a_1 k - a_2) T(n-1,k),
\]
and consider the combinatorial differential equation system
\begin{align*}
	\left\{
	\begin{array}{cl}
		U' = & U V^{\,a_1 + a_2}, \quad U(0)=u, \\
		V' = & V^{\,a_2 + 1}, \quad V(0)=v.
	\end{array}
	\right.
\end{align*}
Then $\suit{T}$ represents the number of $\mathcal{U}$-structures with $k$ connected components.
\end{corollary}

\begin{theorem}\label{ffromgram}
Let $\big(\suit{F}\big)$ be a sequence satisfying \eqref{recu} with $\coul{a} = a_2 n + a_1 k + a_0$, $\coul{b} = 1$, and usual initial conditions. 
If $a_1 \neq 0$ and $a_2 \neq 0$, then
\begin{equation}\label{fnkfromgram}
	\suit{F} = \frac{1}{a_1^k\, k!} \sum_{j=0}^{k} (-1)^{\,k-j} \binom{k}{j} \prod_{r=1}^{n} (a_0 + a_1 j + r a_2).
\end{equation}
\end{theorem}

\begin{proof}
The grammar $G$, given in \eqref{gram2}, is
\begin{align*}
	G = \{& u \to u v^{\,a_1 + a_2}, \\
	& v \to v^{\,a_2 + 1} \}.
\end{align*}
Consider the system
\[
\begin{cases}
	U' = U V^{\,a_1 + a_2}, \quad U(0)=u, \\
	V' = V^{\,a_2 + 1}, \quad V(0)=v.
\end{cases}
\]
By separation of variables, we get
\[
V(t) = \big(v^{-a_2} - a_2 t\big)^{-1/a_2}, \quad
U(t) = u \exp\!\Big(\frac{(v^{-a_2} - a_2 t)^{-a_1 / a_2} - v^{a_1}}{a_1}\Big),
\]
valid as long as $v^{-a_2} - a_2 t \neq 0$.
Then
\[
U(t) V(t)^{\,a_0 + a_2} 
= u v^{\,a_0 + a_2} 
(1 - a_2 t\, v^{a_2})^{-(a_0 + a_2)/a_2} 
\exp\!\Big(\frac{v^{a_1}}{a_1} \big((1 - a_2 t\, v^{a_2})^{-a_1/a_2} - 1\big)\Big).
\]
\noindent
Expanding the exponential and the binomial, we obtain
\[
U(t) V(t)^{\,a_0 + a_2} = \sum_{k \ge 0} \frac{1}{k!} \left(\frac{v^{a_1}}{a_1}\right)^k 
\sum_{j=0}^{k} (-1)^{\,k-j} \binom{k}{j} (1 - a_2 t\, v^{a_2})^{-(a_0 + a_2 + a_1 j)/a_2}.
\]
\noindent
Moreover, for $n \ge 0$,
\[
(1 - a_2 t\, v^{a_2})^{-\beta} = \sum_{n \ge 0} \left(\prod_{r=1}^{n} (\beta + r - 1)\right) \frac{(a_2 t\, v^{a_2})^n}{n!}.
\]
\noindent
Therefore,
\begin{equation}\label{eq22}
	U(t) V(t)^{\,a_0 + a_2} = \sum_{k \ge 0} \frac{1}{a_1^k k!} \sum_{j=0}^{k} (-1)^{\,k-j} \binom{k}{j} \prod_{r=1}^{n} (a_0 + a_1 j + r a_2) \frac{t^n}{n!}.
\end{equation}
\noindent
On the other hand, Proposition \ref{chendumrabez} ensures
\begin{equation}\label{eq23}
	U(t) V(t)^{\,a_0 + a_2} = Gen(uv^{\,a_0 + a_2}, t) = 
	\sum_{n \ge 0} \left( \sum_{k=0}^{n} T(n,k)\, u\, v^{\,a_2 n + a_1 k + a_2 + a_0} \right) \frac{t^n}{n!}.
\end{equation}
\noindent
Comparing \eqref{eq22} and \eqref{eq23} yields
\[
F(n,k) = \frac{1}{a_1^k k!} \sum_{j=0}^{k} (-1)^{\,k-j} \binom{k}{j} \prod_{r=1}^{n} (a_0 + a_1 j + r a_2).
\]
\end{proof}

\section{Applications}
We have already seen an example involving the \( r \)-Whitney--Eulerian numbers.
We now present several further applications. Some of these yield improved formulations of known results, while others provide new contributions.

\subsection{Descents in Stirling $r$-permutations}

Xiao He (\cite{He2023mthOrderEulerian}) studied $B^{(r)}(n,k)$, the numbers of $r$-permutations on $[n]$ with $k$ descents. He proved that for all $k$ and $n>0$, 
\begin{equation}
B^{(r)}(n,k)=\big(rn-k+(1-r)\big)B^{(r)}(n-1,k-1) + (k+1)B^{(r)}(n-1,k),
\label{2.4}
\end{equation}
with $B^{(r)}(n,0)=1\mbox{ pour } $n$\geq 1\text{  et }~ B^{(r)}(n,k)=0$ if $n\leq k$ ou $k<0$.

\noindent
First, the grammar of Hao associated is then 
\begin{align*}
G=\{x&\to x^ry, \\
y&\to x^ry \}.
\end{align*}
\begin{definition}[Full $r$-ary Tree]
A \emph{full $r$-ary tree} is a rooted tree in which every internal node has 
exactly $r$ children, where $r \ge 1$ is fixed. 
Nodes with no children are called \emph{leaves}. 
Thus, in a full $r$-ary tree, each internal node has exactly $r$ children 
and each leaf has $0$ children.
\end{definition}
We denote $\mathcal{A}^{(r)}$ the structure associated to full $r$-ary tree. It can be defined also recursively. For $n=0$: $\mathcal{A}^{(r)}=\emptyset$ and for $n\geq 0$: $\mathcal{A}^{(r)}=\big(p,\mathcal{A}^{(r)}_1,\mathcal{A}^{(r)}_2\cdots, \mathcal{A}^{(r)}_r\big),$ with $p$ is a point called root and the $\big(\mathcal{A}^{(r)}_i\big)$ are full $r$-ary trees on $k_i$ points with $\sum_{i=1}^{i=r}k_i=n-1$.

\noindent
Here, we consider also an order on each children and think that the last among the $r$ children is called cadet. 

\begin{center}
\begin{tikzpicture}[scale=1,
	every node/.style={circle, draw, minimum size=4mm, inner sep=1pt, font=\small},
	box/.style={draw, rectangle, minimum width=11cm, minimum height=5cm}
	]
	
	\node[box, anchor=south west] (b1) at (0,0) {};
	
	\node[fill=darkgray!82] (y5a) at (4.45,0.6) {$1$};

	\node[fill=white!70] (x3a) at (3,1.8) {$x$};
	\node[fill=darkgray!75] (x4a) at (1,1.8) {$2$};
	\node[fill=darkgray!70] (x5a) at (5.5,1.8) {$4$};
	\node[fill=darkgray!65] (x6a) at (8.25,1.8) {$3$};
	
	\draw (y5a)--(x3a); 
	\draw (y5a)--(x4a);
	\draw (y5a)--(x5a); 
	\draw (y5a)--(x6a);
	
	\node[fill=white!70] (z3a) at (0.25,3.2) {$x$};
	\node[fill=white!60]  (z4a) at (1,3.2) {$x$};
	\node[fill=white!70] (z5a) at (1.8,3.2) {$x$};
	\node[fill=darkgray!60](z6a) at (2.6,3.2) {$5$};
	
	\draw (x4a)--(z3a); 
	\draw (x4a)--(z4a);
	\draw (x4a)--(z5a); 
	\draw (x4a)--(z6a);

	\node[fill=white!70] (t3a) at (2,4.4) {$x$};
	\node[fill=white!60]  (t4a) at (2.8,4.4) {$x$};
	\node[fill=white!70] (t5a) at (3.6,4.4) {$x$};
	\node[fill=blue!60](t6a) at (4.4,4.4) {$y$};
	
	\draw (z6a)--(t3a); 
	\draw (z6a)--(t4a);
	\draw (z6a)--(t5a); 
	\draw (z6a)--(t6a);

	\node[fill=white!70] (u3a) at (4,3.2) {$x$};
	\node[fill=white!60]  (u4a) at (4.8,3.2) {$x$};
	\node[fill=white!70] (u5a) at (5.6,3.2) {$x$};
	\node[fill=blue!60](u6a) at (6.4,3.2) {$y$};

	\draw (x5a)--(u3a); 
	\draw (x5a)--(u4a);
	\draw (x5a)--(u5a); 
	\draw (x5a)--(u6a);
	
	\node[fill=white!70] (v3a) at (7.8,3.2) {$x$};
	\node[fill=white!60]  (v4a) at (8.6,3.2) {$x$};
	\node[fill=white!70] (v5a) at (9.4,3.2) {$x$};
	\node[fill=blue!60](v6a) at (10.2,3.2) {$y$};
	
	\draw (x6a)--(v3a); 
	\draw (x6a)--(v4a);
	\draw (x6a)--(v5a); 
	\draw (x6a)--(v6a);
\end{tikzpicture}
\captionof{figure}{A full $4$-ary true for $n=5$}\label{4ary}
\end{center} 
\noindent
In Figure~\ref{4ary}, the tree has 16 leaves and the blue nodes ($y$) are cadets; there are 3 cadets.

\begin{prop}
$B^{(r)}(n,k)$ counts the number of $\mathcal{A}^{(r+1)}$-structures on $[n]$ with $k+1$ cadet leaves.
\end{prop}
\begin{proof}
First, the proposition \ref{phaores} ensures us 
\begin{equation}\label{key1}
	\opd{G}^{n}(y)=\sum\limits_{k=0}^{n}B^{(r)}(n,k)x^{nr-k}y^{k+1}.
\end{equation} Then $B^{(r)}(n,k)$ is the numbers of $\opd{G}^{n}(y)$-structure having $k+1$ $y$-leaves.
We notice also that  $\opd{G}^{n}(y)$-structure is the same than $\mathcal{A}^{(r+1)}$-structure where the leaves are labeled $x$ and $y$ (only the last children is $y$). Thus we conclude. 
\end{proof}

\noindent Moreover, the  system of differential equations associated to the grammar  is
\begin{equation}
\begin{cases}
	X'(t)= X^r(t)Y(t),&X(0)=x \\ 
	Y'(t)= X^r(t)Y(t),&Y(0)=y.
\end{cases}
\label{key2}
\end{equation}
\noindent We apply know Proposition \ref{chendumrabez} (or Theorem \ref{rabeza})  to the equation \eqref{key1} to get the following.
\begin{theorem}\label{theopribb}
If $\big(X(t),Y(t)\big)$ is the solution of the system of differential equations \eqref{key2}, then
\begin{equation}
	\sum_{n\geq 0} \big( \sum\limits_{k=0}^{n}B^{(r)}(n,k)x^{nr-k}y^{k+1} \big) \vatn= Y(t).
\end{equation}
\end{theorem} 
\subsubsection{Case where $r=2$}
We write $B$ for $B^{(2)}$, so \eqref{2.4}'s version is
\begin{equation}\label{5}
\suit{B}=(k+1)B(n-1,k)+(2n-k-1)B(n-1,k-1).
\end{equation}
\noindent
Define the formal series
\[T(z):=\sum_{n\geq 1} n^{n-1}\varn.\]
One can compute, or use directly the fact about nonplane labelled trees in Flajolet(\cite{FlajoletSedgewick2007}) to get
\begin{equation}
T(z)=ze^{T(z)} \implies T'(z)=\frac{e^{T(z)}}{1-zT(z)}=\frac{e^{T(z)}}{1-T(z)}.
\end{equation}
Then, for $x\not=y$
\begin{equation}
\begin{cases}
	X(t)= \dfrac{x-y}{1-T \!\left(
		\frac{y}{x}\exp\!\left(
		-\frac{y}{x} + (x- y)^2 t
		\right)
		\right)}, \\ 
	Y(t)=\dfrac{(x-y)\,T\!\left(
		\frac{y}{x}\exp\!\left(
		-\frac{y}{x} + (x-y)^2 t
		\right)
		\right)}{
		1 - T\!\left(
		\frac{y}{x}\exp\!\left(
		-\frac{y}{x} + (x-y)^2 t
		\right)
		\right)}
\end{cases}
\label{key112}
\end{equation}
solves the equation \eqref{key2} if $r=2$. We apply now the proposition \ref{chendumrabez}, with equation \eqref{key1} to have
\begin{theorem} For $x\not=y$
\begin{equation}\sum_{n\geq 0} \sum_{k=0}^{k=n} B(n,k)x^{2n-k}y^{k+1} \vatn=\dfrac{(x-y)\,T\!\left(
		\frac{y}{x}\exp\!\left(
		-\frac{y}{x} + (x-y)^2 t
		\right)
		\right)}{
		1 - T\!\left(
		\frac{y}{x}\exp\!\left(
		-\frac{y}{x} + (x-y)^2 t
		\right)
		\right)}. \end{equation}
		\end{theorem}
		\begin{corollary}
For $y\not=1$
\begin{equation}\sum_{n\geq 0} \sum_{k=0}^{k=n} B(n,k)y^{k+1} \vatn=\dfrac{(1-y)\,T\!\left(
		y\exp\!\left(
		-y + (1-y)^2 t
		\right)
		\right)}{
		1 - T\!\left(
		y\exp\!\left(
		-y + (1-y)^2 t
		\right)
		\right)}. \end{equation}
		\end{corollary}
		
		\subsubsection{General case $r\geq 1$}
		Let's define first a combinatorial structure.
		\begin{definition}\label{combistruc}
Fix integers $a, b, r$ with $b-a\in \mathbb{N}$.
\begin{itemize}
	\item A tree may consist of a single leaf. Such a leaf can be colored in $a$ different colors.
	
	\item Otherwise, a leaf can be expanded into an internal node. Each internal node produces either:
	\begin{itemize}
		\item exactly $r$ children, or
		\item exactly $r+1$ children.
	\end{itemize}
	
	\item Coloring rules:
	\begin{itemize}
		\item Leaves are colored with $a$ colors;
		\item Internal nodes with $r$ children can be colored in $(b-a)$ different ways;
		\item Internal nodes with $r+1$ children carry no additional color.
	\end{itemize}
	\item we label each internal node in such way it is croissant (increasing) meanly if node x is descend of node y then labbel of x > label of y.
\end{itemize}
\end{definition}

\noindent Let $c_n$ be the numbers of such tree with $n$ internal nodes and $C^a_b(t):=\sum_{n\geq 1} c_n\frac{t^n}{n!}$ be the generating function. One can see is 
\begin{equation}
\begin{cases}
	c_0=a \\ 
	c_{k+1}
	=
	\sum_{\substack{k_1+\cdots+k_{r+1}=k}}
	\binom{k}{k_1,\dots,k_{r+1}} c_{k_1}\cdots c_{k_{r+1}}
	+
	(b-a)
	\sum_{\substack{k_1+\cdots+k_r=k}}
	\binom{k}{k_1,\dots,k_r} c_{k_1}\cdots c_{k_r}.
\end{cases}
\label{key11}
\end{equation}
\begin{definition} \label{serieb}
We define the formal series
\[C_y(t):=\sum_{n\geq 1} f_n(y)\frac{t^n}{n!},\]
where
\begin{equation}
	\begin{cases}
		f_0=1 \\ 
		f_{k+1}(y)
		=
		\sum_{\substack{k_1+\cdots+k_{r+1}=k}}
		\binom{k}{k_1,\dots,k_{r+1}} f_{k_1}(y)\cdots f_{k_{r+1}}(y)
		+
		(y-1)
		\sum_{\substack{k_1+\cdots+k_r=k}}
		\binom{k}{k_1,\dots,k_r} f_{k_1}(y)\cdots f_{k_r}(y).
	\end{cases}
	\label{key22}
\end{equation}
\end{definition}

\noindent Analytically, we can check that
\begin{equation}
C_y'(t)= C_y(t)^{r+1}+(y-1) C_y(t)^r; \quad C_y(0)=1.
\end{equation}
The construction is just $C_y=C_y^1$. 

\noindent The combinatorial integration with the definition \ref{combistruc} also allow us to say
\begin{equation}\label{key33}
(\mathcal{C}_b^a)'=(\mathcal{C}_b^a)^{r+1}+(b-a) (\mathcal{C}_b^a)^{r}\implies ({C}_b^a)'(t)=({C}_b^a)^{r+1}(t)+(b-a) ({C}_b^a)^{r}(t).
\end{equation} 
So, $\big( C_y^x(t), C_y^x(t)+(y-x)   \big)$ solves the system of differential equations \eqref{key2}. Thus, Theorem \ref{theopribb} gives us the following.
\begin{theorem}
\begin{equation}
	\sum_{n\geq 0} \big( \sum\limits_{k=0}^{n}B^{(r)}(n,k)x^{nr-k}y^{k+1} \big) \vatn= C_y^x(t)+(y-x).
\end{equation}
\end{theorem}
\noindent In particular
\begin{theorem}
\begin{equation}
	\sum_{n\geq 0} \big( \sum\limits_{k=0}^{n}B^{(r)}(n,k)y^{k+1} \big) \vatn= C_y(t)+(y-1).
\end{equation}
\end{theorem}

\subsection{Case \( b_2 = 0 \)}
Many years ago, Wilf \cite{wilf2004methodcharacteristicsproblem89} (and also Neuwirth \cite{NEUWIRTH200133}) found the generating function for the case $b_2=0$, which yields an explicit formula. We will show here that these results can be seen quickly using grammar.

The Hao grammar associated is $u\to u^{b_1+1}v^{a_1+a_2}$, $v\to v^{a_2+1}$ with differential system
\[
\left\{
\begin{array}{ll}
U' = U^{b_1+1} V^{a_1+a_2}, & U(0)=u,\\
V' = V^{a_2+1}, & V(0)=v.
\end{array}
\right.
\]

\renewcommand{\arraystretch}{2.3}

\begin{center}
\resizebox{1\linewidth}{!}{$
	\begin{array}{|c|c|c|c|}
		\hline
		& b_1 = 0 & b_1 = -1 & b_1 \neq 0,-1 \\
		\hline
		
		a_2 = 0 &
		\begin{array}{c}
			V(t) = v e^t\\
			U(t) =
			\begin{cases}
				u \exp\!\left(\frac{v^{a_1}}{a_1}(e^{a_1 t}-1)\right),\\
				\hfill a_1 \neq 0\\
				u e^t, \quad a_1 = 0
			\end{cases}
		\end{array}
		&
		\begin{array}{c}
			V(t) = v e^t\\
			U(t) =
			\begin{cases}
				u + \frac{v^{a_1}}{a_1}(e^{a_1 t}-1),\\
				\hfill a_1 \neq 0\\
				u + t, \quad a_1 = 0
			\end{cases}
		\end{array}
		&
		\begin{array}{c}
			V(t) = v e^t\\
			U(t) =
			\begin{cases}
				\left[u^{-b_1} - \frac{b_1 v^{a_1}}{a_1}(e^{a_1 t}-1)\right]^{-1/b_1},\\
				\hfill a_1 \neq 0\\
				\left[u^{-b_1} - b_1 t\right]^{-1/b_1}, \quad a_1 = 0
			\end{cases}
		\end{array}
		\\
		\hline
		
		a_2 = -1 &
		\begin{array}{c}
			V(t) = v+t\\
			U(t) =
			\begin{cases}
				u \exp\!\left(\frac{(v+t)^{a_1}-v^{a_1}}{a_1}\right),\\
				\hfill a_1 \neq 0\\
				u \frac{v+t}{v}, \quad a_1 = 0
			\end{cases}
		\end{array}
		&
		\begin{array}{c}
			V(t) = v+t\\
			U (t)=
			\begin{cases}
				u + \frac{(v+t)^{a_1}-v^{a_1}}{a_1},\\
				\hfill a_1 \neq 0\\
				u + \ln\!\left(\frac{v+t}{v}\right), \quad a_1 = 0
			\end{cases}
		\end{array}
		&
		\begin{array}{c}
			V(t) = v+t\\
			U(t) =
			\begin{cases}
				\left[u^{-b_1} - \frac{b_1}{a_1}\big((v+t)^{a_1}-v^{a_1}\big)\right]^{-1/b_1},\\
				\hfill a_1 \neq 0\\
				\left[u^{-b_1} - b_1 \ln\!\left(\frac{v+t}{v}\right)\right]^{-1/b_1}, \quad a_1 = 0
			\end{cases}
		\end{array}
		\\
		\hline
		
		a_2 \neq 0,-1 &
		\begin{array}{c}
			V(t) = (v^{-a_2}-a_2 t)^{-1/a_2}\\
			U(t) =
			\begin{cases}
				u \exp\!\left(\frac{(v^{-a_2}-a_2 t)^{-a_1/a_2} - v^{a_1}}{-a_1}\right),\\
				\hfill a_1 \neq 0\\
				u (v^{-a_2}-a_2 t)^{-1/a_2} v, \quad a_1 = 0
			\end{cases}
		\end{array}
		&
		\begin{array}{c}
			V(t) = (v^{-a_2}-a_2 t)^{-1/a_2}\\
			U(t) =
			\begin{cases}
				u + \frac{(v^{-a_2}-a_2 t)^{-a_1/a_2} - v^{a_1}}{-a_1},\\
				\hfill a_1 \neq 0\\
				u - \frac{1}{a_2}\ln\!\left(\frac{v^{-a_2}-a_2 t}{v^{-a_2}}\right), \quad a_1 = 0
			\end{cases}
		\end{array}
		&
		\begin{array}{c}
			V(t) = (v^{-a_2}-a_2 t)^{-1/a_2}\\
			U(t) =
			\begin{cases}
				\left[u^{-b_1} - \frac{b_1}{a_1}
				\left((v^{-a_2}-a_2 t)^{-a_1/a_2} - v^{a_1}\right)\right]^{-1/b_1},\\
				\hfill a_1 \neq 0\\
				\left[u^{-b_1} + \frac{b_1}{a_2}\ln\!\left(\frac{v^{-a_2}-a_2 t}{v^{-a_2}}\right)\right]^{-1/b_1}, \quad a_1 = 0
			\end{cases}
		\end{array}
		\\
		\hline
		
	\end{array}
	$}
	\end{center}
	\noindent This can be combine with Propositions \ref{phaores} and \ref{chendumrabez} one 
	\[U(t)^{b_1+b_0}V(t)^{a_2+a_0}=\sum_{n\geq 0} \Big( \sum_{k=0}^n \suit{T} u^{b_1 k +b_0+b_1} v^{a_2 n + a_1 k +a_0+a_2}\Big) \frac{t^n}{n!}.\]
	\noindent For instance, if we take the case where $a_2,b_1\notin \{0,-1\}$, $a_1\not=0$ for $v=1$:
	\[
	\left( u^{-b_1} + \frac{b_1}{a_1} \left( 1 - (1 - a_2 t)^{-a_1/a_2} \right) \right)^{-\frac{b_1+b_0}{b_1}} 
	\cdot
	(1 - a_2 t)^{-\frac{a_2 + a_0}{a_2}} = \sum_{n\geq 0} \Big( \sum_{k=0}^n \suit{T} u^{b_1 k +b_0+b_1} \Big) \frac{t^n}{n!}.
	\]
	We can change now $x=u^{b_1}$ and get the following.
	\begin{prop} If  $a_2,b_1\notin \{0,-1\}$, $a_1\not=0$, then
\[
\left( 1 + x \frac{b_1}{a_1} \left( 1 - (1 - a_2 t)^{-a_1/a_2} \right) \right)^{-\frac{b_1+b_0}{b_1}}
(1 - a_2 t)^{-\frac{a_2+a_0}{a_2}}
=
\sum_{n\ge 0} \left( \sum_{k=0}^n T(n,k)\, x^k \right)\frac{t^n}{n!}.
\]
\end{prop}
This is similar to what Herbert \cite{wilf2004methodcharacteristicsproblem89} has. 

\subsection{Case \( a_1+a_2 = 0 \)}
The Hao grammar associated is $u\to u^{b_1+b_2+1}$, $v\to u^{b_2}v^{a_2+1}$ with differential system
\[
\left\{
\begin{array}{ll}
U' = U^{b_1+b_2+1}, & U(0)=u,\\
V' = U^{b_2}V^{a_2+1}, & V(0)=v.
\end{array}
\right.
\]

\renewcommand{\arraystretch}{2.3}
\begin{center}
\resizebox{1\linewidth}{!}{$
	\begin{array}{|c|c|c|c|}
		\hline
		& a_2 = 0 & a_2 = -1 & a_2 \neq 0,-1 \\
		\hline
		
		b_1+b_2 = 0 &
		\begin{array}{c}
			U(t) = u e^t\\
			V(t) =
			\begin{cases}
				v \exp\!\left(\frac{u^{b_2}}{b_2}(e^{b_2 t}-1)\right),\\
				\hfill b_2 \neq 0\\
				v e^t, \quad b_2 = 0
			\end{cases}
		\end{array}
		&
		\begin{array}{c}
			U(t) = u e^t\\
			V(t) =
			\begin{cases}
				v + \frac{u^{b_2}}{b_2}(e^{b_2 t}-1),\\
				\hfill b_2 \neq 0\\
				v + t, \quad b_2 = 0
			\end{cases}
		\end{array}
		&
		\begin{array}{c}
			U(t) = u e^t\\
			V(t) =
			\begin{cases}
				\left[v^{-a_2} - \frac{a_2 u^{b_2}}{b_2}(e^{b_2 t}-1)\right]^{-1/a_2},\\
				\hfill b_2 \neq 0\\
				\left[v^{-a_2} - a_2 t\right]^{-1/a_2}, \quad b_2 = 0
			\end{cases}
		\end{array}
		\\
		\hline
		
		b_1+b_2 = -1 &
		\begin{array}{c}
			U(t) = u + t \\
			V(t) =
			\begin{cases}
				v \exp\!\left(\frac{(u+t)^{b_2+1}-u^{b_2+1}}{b_2+1}\right),\\
				\hfill b_2 \neq -1\\
				v \frac{u+t}{u}, \quad b_2 = -1
			\end{cases}
		\end{array}
		&
		\begin{array}{c}
			U(t) = u + t\\
			V(t) =
			\begin{cases}
				v + \frac{(u+t)^{b_2+1}-u^{b_2+1}}{b_2+1},\\
				\hfill b_2 \neq -1\\
				v + \ln\!\left(\frac{u+t}{u}\right), \quad b_2 = -1
			\end{cases}
		\end{array}
		&
		\begin{array}{c}
			U(t) = u + t \\
			V(t) =
			\begin{cases}
				\left[v^{-a_2} - \frac{a_2}{b_2+1}\big((u+t)^{b_2+1}-u^{b_2+1}\big)\right]^{-1/a_2},\\
				\hfill b_2 \neq -1\\
				\left[v^{-a_2} - a_2 \ln\!\left(\frac{u+t}{u}\right)\right]^{-1/a_2}, \quad b_2 = -1
			\end{cases}
		\end{array}
		\\
		\hline
		
		b_1+b_2 \neq 0,-1 &
		\begin{array}{c}
			U(t) = (u^{-(b_1+b_2)}-(b_1+b_2)t)^{-1/(b_1+b_2)}\\
			V(t) =
			\begin{cases}
				v \exp\!\left(
				\frac{
					(u^{-(b_1+b_2)}-(b_1+b_2)t)^{1-\frac{b_2}{b_1+b_2}}
					- u^{b_2}
				}{
					b_1
				}
				\right),\\
				\hfill b_1 \neq 0\\
				v \exp\!\left(\ln\!\frac{u^{-(b_1+b_2)}-(b_1+b_2)t}{u^{-(b_1+b_2)}}\right), \quad b_1=0
			\end{cases}
		\end{array}
		&
		\begin{array}{c}
			U(t) = (u^{-(b_1+b_2)}-(b_1+b_2)t)^{-1/(b_1+b_2)}\\
			V(t) =
			\begin{cases}
				v + 
				\frac{
					(u^{-(b_1+b_2)}-(b_1+b_2)t)^{1-\frac{b_2}{b_1+b_2}}
					- u^{b_2}
				}{
					b_1
				},\\
				\hfill b_1 \neq 0\\
				v + \ln\!\left(\frac{u^{-(b_1+b_2)}-(b_1+b_2)t}{u^{-(b_1+b_2)}}\right), \quad b_1=0
			\end{cases}
		\end{array}
		&
		\begin{array}{c}
			U(t) = (u^{-(b_1+b_2)}-(b_1+b_2)t)^{-1/(b_1+b_2)}\\
			V(t)=
			\begin{cases}
				\left[
				v^{-a_2}
				+
				\frac{a_2}{b_1}\big((u^{-(b_1+b_2)}-(b_1+b_2)t)^{b_1/(b_1+b_2)}-u^{-b_1}\big)
				\right]^{-1/a_2}, \quad b_1 \neq 0,\\[10pt]
				\left[
				v^{-a_2}
				- a_2 \ln\!\left(\dfrac{(u^{-(b_1+b_2)}-(b_1+b_2)t)^{-1/(b_1+b_2)}}{u}\right)
				\right]^{-1/a_2}, \quad b_1=0.
			\end{cases}
		\end{array}
		\\
		\hline
		
	\end{array}
	$}
	\end{center}
	\noindent This can be combine with Propositions \ref{phaores} and \ref{chendumrabez} one 
	\[U(t)^{b_2+b_1+b_0}V(t)^{a_2+a_0}=\sum_{n\geq 0} \Big( \sum_{k=0}^n \suit{T} u^{b_2n+b_1 k +b_0+b_1+b_2} v^{a_2 n + a_1 k +a_0+a_2}\Big) \frac{t^n}{n!}.\]
	\noindent For instance, if we take the case where $a_2,b_1+b_2\notin \{0,-1\}$, $b_1\not=0$ for $u=1$ and 	$y=v^{a_1}, x = y^{a_2/a_1}t$,:
	
	\[	\begin{aligned}
&\left(1-(b_1+b_2)t\, y\right)^{-(b_1+b_2+b_0)/(b_1+b_2)}
\\
&\quad \cdot
\left[
y
+
\frac{a_2}{b_1}
\left(
\left(1-(b_1+b_2)t\, y\right)^{b_1/(b_1+b_2)} - 1
\right)
\right]^{-(a_2+a_0)/a_2}
\\
&=
y^{(a_0+a_2)/a_1}
\sum_{n\ge 0}
\frac{t^n}{n!}
\sum_{k=0}^n T(n,k) y^k.
\end{aligned}\]

\begin{prop} If  $a_2,b_1+b_2\notin \{0,-1\}$, $b_1\not=0$, then
\[
\begin{aligned}
	\sum_{n\ge 0}\sum_{k=0}^n T(n,k)\, y^k \frac{t^n}{n!}
	&=
	y^{-\frac{a_0+a_2}{a_1}}
	\left(1-(b_1+b_2)t\, y^{-\frac{a_2}{a_1}}\right)^{-\frac{b_0+b_1+b_2}{b_1+b_2}} \\
	&\quad \times
	\left[
	y
	+ \frac{a_2}{b_1}
	\left(
	\left(1-(b_1+b_2)t\, y\right)^{\frac{b_1}{b_1+b_2}} - 1
	\right)
	\right]^{-\frac{a_0+a_2}{a_2}}\\
	&=
	\left(1-(b_1+b_2)t\, y\right)^{-\frac{b_0+b_1+b_2}{b_1+b_2}}
	\left[
	1 + \frac{a_2}{b_1y}\, 
	\left(
	\left(1-(b_1+b_2)t\, y\right)^{\frac{b_1}{b_1+b_2}} - 1
	\right)
	\right]^{-\frac{a_0+a_2}{a_2}}.
\end{aligned}
\]

	%
	%

\end{prop}

\end{document}